\documentclass[final]{siamart171218}

\usepackage{lipsum}
\usepackage{amsfonts,amsmath}
\usepackage{lineno}
\usepackage{graphicx}
\modulolinenumbers[5]
\usepackage{algorithmic}
\usepackage[color]{showkeys}
\definecolor{refkey}{gray}{.75}
\usepackage{booktabs}

\usepackage{scalerel,stackengine}
\stackMath
\newcommand\reallywidehat[1]{%
\savestack{\tmpbox}{\stretchto{%
  \scaleto{%
    \scalerel*[\widthof{\ensuremath{#1}}]{\kern-.6pt\bigwedge\kern-.6pt}%
    {\rule[-\textheight/2]{1ex}{\textheight}}%WIDTH-LIMITED BIG WEDGE
  }{\textheight}% 
}{0.5ex}}%
\stackon[1pt]{#1}{\tmpbox}%
}
\def\p{\partial}
\def\tilde{\widetilde}
\def\hat{\widehat}
\def\bs{\boldsymbol}

\def\nb{\mathbf{n}}
\def\ub{\bs{u}}

\def\kb{\bs{\mathrm{k}}}

\def\tb{\mathbf{t}}

\def\CE{\mathcal{CE}}
\def\VE{\mathcal{VE}}
\def\VC{\mathcal{VC}}
\def\CV{\mathcal{CV}}
\def\EC{\mathcal{EC}}
\def\EV{\mathcal{EV}}
\def\VE{\mathcal{VE}}
\def\BC{\mathcal{BC}}
\def\IC{\mathcal{IC}}
\def\BE{\mathcal{BE}}
\def\IE{\mathcal{IE}}

\def\gradh{\nabla_h}
\def\sgradh{\nabla^\perp_h}

\def\V{\mathcal{V}}
\def\E{\mathcal{E}}

\def\C{\mathcal{C}}

\def\M{\mathcal{M}}

\newcommand{\ra}[1]{\renewcommand{\arraystretch}{#1}}

\ifpdf
  \DeclareGraphicsExtensions{.eps,.pdf,.png,.jpg}
\else
  \DeclareGraphicsExtensions{.eps}
\fi

\numberwithin{theorem}{section}

\newcommand{\TheTitle}{Conservative numerical schemes with optimal
  dispersive wave relations --- Part
  I.~Derivations and Analyses}  
\newcommand{\TheAuthors}{Qingshan Chen, Lili Ju, and Roger Temam}
\newcommand{\ShortTitle}{Energy and enstrophy conservations}

\headers{\ShortTitle}{\TheAuthors}

\title{{\TheTitle}\thanks{Submitted to the editors DATE.
\funding{This work was in part supported by ??}}} 

\author{
 Qingshan Chen\thanks{Corresponding author: Department of Mathematical Sciences, Clemson
   University, Clemson, SC 29631.
   (\email{qsc@clemson.edu}, \url{http://mathfluids.github.io}).},
\and 
 Lili Ju
 \and
  Roger Temam
 }

\usepackage{amsopn}

\ifpdf
\hypersetup{
  pdftitle={\TheTitle},
  pdfauthor={\TheAuthors}
}
\fi

\begin{document}
\maketitle

\begin{abstract}
An energy-conserving
 and an energy-and-enstrophy conserving numerical schemes  are
derived, by approximating the Hamiltonian formulation, based on the
Poisson brackets and the vorticity-divergence variables, of the
inviscid shallow water flows. The conservation of the energy and/or
enstrophy stems from skew-symmetry of the Poisson brackets, which is
retained in the discrete approximations. These
schemes operate on unstructured orthogonal dual meshes, over bounded
or unbounded domains, and they are also shown to possess the same optimal
dispersive wave relations as those of the Z-grid scheme. 
\end{abstract}

\begin{keyword}
Unstructured meshes, shallow water equations, large-scale geophysical flows,
energy conservation, enstrophy conservation, Hamiltonian structures,
Poisson bracket, dispersive wave
relations
\end{keyword}

% REQUIRED
\begin{AMS}
35Q86, 65M08, 65P10
\end{AMS}

\section{Introduction}\label{sec:intro}
% Long-term weather and climate forecasts remain a grand challenge. One
% of the main reasons for this challenge
Long-term weather and climate forecasts remain an challenging
enterprise. There are many factors contributing to this challenge, and
prominently amidst them
is the fact that two major
components of the weather and climate systems, namely the ocean and
atmosphere, are inherently multi-scale, nonlinear, turbulent, and
chaotic. For flows of this nature, accurately reproducing their
point-wise behaviors over long-period simulations is out of
question. But point-wise behaviors are of little interest for
long-term weather and climate forecasts; of far more importance  are
certain aggregated statistics, such as the annual mean precipitation,
likelihood of a drought, etc. How to improve the fidelity of weather
and climate models in predicting these statistics is a wide open
question with no definitive answers in sight yet. But it is not hard
to come up with some obviously necessary conditions for these models
to be successful. While the ocean and atmosphere are
chaotic in nature, they also very predictably possess some
characteristic structures that distinguish them from other physical or
engineering flows, such as the gravity-induced stratification and the
geostrophic balance between the Coriolis force and the horizontal
pressure gradient. Models of these flows, for any chance of success,
must reliably reproduce these characteristic structures, for, without
them, the model lose any relevance to the large-scale geophysical
flows, and can not be trusted to convey any reliable statistics about
them. The ocean and atmosphere are close to being inviscid in the
interior. When the viscosity is ignored, large-scale
geophysical flows conserve many quantities, such as the mass,
vorticity, energy, etc. It is then reasonable to expect that models
of these flows, under the same assumption, should also conserve at
least a select set of these quantities. A model with only half of the
mass or energy left at the end of a simulation is unlikely a good
approximation of the reality.

This work focuses on these two issues of model developments, namely
the preservation of the characteristic structures and the conservation of key
quantities. On the first one, only the preservation of the geostrophic
balance will be pursued here. The issue of
stratification will be studied 
in future works.

The Coriolis force, which is a result of the earth rotation, and the
horizontal pressure gradient are two dominant terms in the dynamical
equations for both the ocean and atmosphere. The two dominant forces
must approximately strike a balance, and this balance is called the
geostrophic balance. Oceanic and atmospheric flows constantly evolve around the
geostrophic balance. The process through which geophysical
flows always manage to keep themselves in the vicinity of a perfect
geostrophic balance is called geostrophic adjustment. It was
recognized long time ago (\cite{Arakawa1977-og}) that a key element of
this process is the dispersion of the internal waves. Thus, the
capability of a model to maintain the geostrophic balance is directly
related to its representation of the dispersive wave relations within
the flows. Again, it was discovered long time ago that the placement
of discrete variables have an impact on the representation of discrete
wave relations. Arakawa and Lamb (\cite{Arakawa1977-og}) studied a
suite of staggering techniques, labeled by single letters from A to E,
and demonstrated through analysis that each of these techniques
provide a different representation of the dispersive relations. Among
these staggering techniques, the so-called C-grid scheme appears to
be the most popular choice, thanks to its decent dispersive wave
relations when the Rossby deformation radius is adequately
resolved. Randall (\cite{Randall1994-vu}) proposed an entirely new
staggering, actually non-staggering technique, called the Z-grid
scheme, and it is based on the
vorticity-divergence formulation of the dynamical equations. It was
shown that the Z-grid scheme possesses the 
optimal representation of the dispersive wave relations among
second-order numerical schemes, regardless of
whether the Rossby deformation radius is resolved or not.

The capability of a numerical model to conserve key quantities is
severely limited by its finite dimensionality. 
% The finite-dimensionality of a discrete system severely limits its
% capability to conserve quantities that are conserved in the analytical
% system.
An analytical system, with an infinite number of degrees of freedom,
can conserve many, or even infinitely
many quantities. A finite-dimensional discrete system often struggles
to conserve more than a few of these quantities. Energy and enstrophy
are two key quantities for large-scale geophysical flows. In the
absence of 
the diffusion and the external forcing, both energy and enstrophy are
conserved under the analytical system. % But many numerical scheme fail
% to conserve both of these quantities. For example, the TRiSK scheme
% (\cite{Ringler2010-sm}) can be configured to conserve either energy,
% or enstrophy, but not bot. The co-volume scheme by one of the authors
% of this article conserves only the energy (\cite{Chen2013-fa}). A few
% schemes did manage to conserve both quantities.
But conserving both quantities in  discrete systems is not a trivial
task. 
By incorporating
conservations as constraints directly into the derivation process,
Arakawa and Lamb (\cite[AL81]{Arakawa1981-dy}) obtained a numerical scheme
that conserve both energy and enstrophy for the shallow water
equations (SWEs) over rectangular and spherical meshes. A
fourth-order-accurate version of the Arakawa and Lamb scheme was
presented in (\cite{Takano1982-dz}). In a series of papers, Salmon
(\cite{Salmon2004-tt, Salmon2005-pd, Salmon2007-dm, Salmon2009-xr})
utilized the Hamiltonian formulation to obtain numerical schemes that
conserve both energy and enstrophy over rectangular meshes. McRae and
Cotter (\cite[MC]{McRae2014-db}) derived an energy and enstrophy
conserving 
scheme using the mixed finite element method.  With the mixed finite
element method and the Hamiltonian formulation, Eldred and Randall
(\cite{Eldred2017-ji}) extended the AL81 scheme onto unstructured
meshes over the global sphere, and also obtained a generalization of
the Z-grid scheme, which conserves both energy and enstrophy. Bauer
and Cotter (\cite{Bauer2018-uv}) extended the MC scheme onto bounded
domain, adding the potential vorticity as a prognostic variable on the
boundary.

We aim to develop a numerical scheme that conserves both energy and
enstrophy, and possesses good
dispersive wave relations. The
scheme should be able to operate over both bonded and unbounded
domains, so that it is applicable to both the ocean and atmosphere. To
achieve these goals, we adopt the vorticity-divergence formulation of
the dynamical equations, in the hope that the scheme will inherit the
optimal dispersive wave relations of the Z-grid scheme of
\cite{Randall1994-vu}). For the conservations, we adopt the
Hamiltonian approach, which has a long history of success in ensuring
the conservation of certain key quantities 
(\cite{Feng1986-va, Wang2001-hv, Gassmann2008-dm}).

In the Hamiltonian approach, the energy conservation is tied with
the skew-symmetry of the Poisson bracket, and the enstrophy
conservation is a consequence of the fact that the quantity is a
singularity to the Poisson bracket. In the discrete approximations, as
it turns out, preserving the skew-symmetry of the Poisson 
bracket, and thus the energy conservation, is straightforward, but
ensuring that the enstrophy remains a singularity to the Poisson
bracket is not. To overcome this difficulty, Salmon
(\cite{Salmon2005-pd}) replaces the bilinear Poisson bracket with the trilinear
Nambu bracket, which includes the potential enstrophy as a third
input parameter, and is skew-symmetric with regard to all of its
parameters. This goes well on unbounded domains, since, in the absence
of boundaries, the Poisson bracket and the Nambu bracket are
equivalent. But they are not over bounded domains. It is discovered
in this study that, while they are not equivalent {\itshape per se},
the vorticity-related components of these brackets are actually
equivalent even over bounded domains, thanks to the homogeneous
Dirichlet boundary conditions commonly prescribed on the
streamfunction. Thus, in departure from the common practice in the
literature, we replace the vorticity-related component of the Poisson
bracket by its Nambu counterpart, and keeps the rest of the Poisson
bracket intact. 

This project uses unstructured meshes, specifically the centroidal
Voronoi tessellations (\cite{Du1999-th, Du2003-gn}) for the model
development. Traditionally, structured meshes dominate this field
(\cite{Bryan1968-dd, Griffies2000-jo, Smith2010-tj}. In recent years,
however, unstructured meshes are becoming popular (\cite{Pain2005-ug,
  Skamarock2012-dt, Ringler2013-pj}), for their capability to maintain
high qualities across the entire domain, even at very high
resolutions, making themselves an appealing option in the age of
exascale computing. Another advantage of unstructured meshes is that
they can provide an accurate representation of the coastal lines in
the case of ocean modeling.

The rest of the article is organized as follows. Section
\ref{sec:analytic} presents a review of the analytic system of the
shallow water equations and its Hamiltonian formulation. Section
\ref{sec:discrete} details the derivation, through the Hamiltonian
approach,  a numerical scheme that conserves energy only, and a numerical
scheme that conserves both the energy and the enstrophy. Section
\ref{sec:linear-analysis} deals with the dispersive wave relations of
the newly derived schemes. Section \ref{sec:conclusion} concludes with
a few remarks.

\section{The continuous system and its Hamiltonian
  formulation}\label{sec:analytic} 
The shallow water equations (SWEs)for large-scale rotating geophysical flows
read as
\begin{equation}
  \label{eq:1}
  \left\{
    \begin{aligned}
      &\dfrac{\p}{\p t}\phi + \nabla\cdot(\phi\ub) = 0,\\
      &\dfrac{\p}{\p t}\ub + \ub\cdot\nabla\ub + f\kb\times\ub =
      -\nabla\left(g(\phi+b)\right),
    \end{aligned}\right.
\end{equation}
where $\phi$ represents the fluid thickness, $\ub$ the horizontal
velocity field of the flow, $b$ the topography, $f$ the Coriolis force
due to the earth rotation, $\kb$ the unit vector pointing in the local
vertical direction, and finally $g$ the gravitational acceleration. 
In geoscience, the so-called vector invariant form of the system is
often preferred for its independence from the choice of the coordinate
system and it given by
\begin{equation}
  \label{eq:2}
  \left\{
    \begin{aligned}
      &\dfrac{\p}{\p t}\phi + \nabla\cdot(\phi\ub) = 0,\\
      &\dfrac{\p}{\p t}\ub + q\kb\times(\phi\ub) =
      -\nabla\left(g(\phi+b) + K\right),
    \end{aligned}\right.
\end{equation}
where $q\equiv (f+\nabla\times\ub)/\phi$ stands for the
potential vorticity, and $K \equiv |\ub|^2/2$ the kinetic energy of
the flow per unit volume.

Taking curl and divergence of the momentum equation of \eqref{eq:2},
one obtains the vorticity-divergence formulation of the SWEs as
\begin{equation}
  \label{eq:3}
  \left\{
    \begin{aligned}
      &\dfrac{\p}{\p t}\phi + \nabla\cdot(\phi\ub) = 0,\\
      &\dfrac{\p}{\p t}\zeta + \nabla\cdot(q\phi\ub) = 0,\\
      &\dfrac{\p}{\p t}\gamma - \nabla\times(q\phi\ub) = 
      -\Delta\left(g(\phi+b) + K\right),
    \end{aligned}\right.
\end{equation}
where $\zeta \equiv \nabla\times\ub$ represents the relative vorticity,
and $\gamma\equiv\nabla\cdot\ub$ the divergence. For a
two-dimensional model, such as the SWEs, both $\zeta$ and $\gamma$ are
scalar quantities, and this is one of the  reasons that  the vorticity-divergence
formulation is preferred on unstructured meshes. 

Equation $\eqref{eq:3}_2$ makes it clear that the relative vorticity
$\zeta$ is not invariant along fluid paths, and thus is not a
tracer-like quantity, in contrast to
two-dimensional incompressible flows. For the SWEs and geophysical
flows in general, that role is played by the potential vorticity $q$
mentioned in the above, for one can easily derive from the first two
equations of \eqref{eq:3} that
\begin{equation}
  \label{eq:4}
   \dfrac{\p}{\p t}q + \ub\cdot \nabla q = 0.
\end{equation}
The potential vorticity (PV) $q$ is a key quantity for large-scale
geophysical flows (\cite{Pedlosky1987-gk}).

We consider the system on a bounded domain $\M$. The flow is assumed
to be inviscid, and obey the no-flux boundary conditions,
\begin{equation}
  \label{eq:5}
  \ub\cdot\nb = 0,\qquad \p\M.
\end{equation}

Under the inviscid SWEs, both the energy and the potential enstrophy
are conserved. For the energy conservation, one multiplies
$\eqref{eq:2}_2$ by $\phi\ub$, $\eqref{eq:2}_1$ by $|\ub|^2/2$, and
adds the resulting equations to obtain
\begin{equation}
  \label{eq:6}
  \dfrac{\p}{\p t}\left(\dfrac{1}{2}\phi|\ub|^2\right) +
  \nabla\cdot\left(\dfrac{1}{2}|\ub|^2\phi\ub\right) =
  -\phi\ub\cdot\nabla\left(g(\phi+b)\right). 
\end{equation}
Integrating this equation over the domain $\M$, and with the no-flux
boundary condition \eqref{eq:5}, one derives the
kinetic energy budget for the shallow water system \eqref{eq:2},
\begin{equation}
  \label{eq:7}
  \dfrac{d}{d t}\int_\M \dfrac{1}{2}\phi|\ub|^2 d{\mathbf x} = 
  -\int_\M \phi\ub\cdot\nabla\left(g(\phi+b)\right) d{\mathbf x}. 
\end{equation}
For the potential energy budget, one starts from $\eqref{eq:2}_1$,
\begin{equation}
  \label{eq:8}
  \dfrac{\p}{\p t}(\phi + b) + \nabla\cdot(\phi\ub) = 0.
\end{equation}
Multiplying both sides by $g(\phi+b)$, and integrating over $\M$, one
obtains
\begin{equation}
  \label{eq:9}
  \dfrac{d}{d t}\int_\M \dfrac{1}{2}g(\phi + b)^2 {\mathbf x} =   \int_\M
  \phi\ub\cdot\nabla\left(g(\phi+b)\right) d{\mathbf x}.  
\end{equation}
The no-flux boundary condition \eqref{eq:5} has been used to rewrite
the right-hand side into the current form. It is apparent that the
terms on the right-hand sides of \eqref{eq:7} and \eqref{eq:9}, being
of the opposite signs, represent the conversion between kinetic and
potential energies within the shallow water system. When added
together, the right-hand sides of these two equations cancel, and one
arrives at
\begin{equation}
  \label{eq:10}
  \dfrac{d}{d t}\int_\M\left(\dfrac{1}{2}\phi|\ub|^2+
    \dfrac{1}{2}g(\phi + b)^2\right) d{\mathbf x} =  0,
\end{equation}
which proves that the total energy of the system, defined as the sum
of the kinetic and potential energies, is conserved. 

The conservation of the potential enstrophy can be established using
the first two equations of \eqref{eq:3} only. Indeed, these two
equations give rise of the transport equation \eqref{eq:4} for the
potential vorticity $q$. Multiplying this equation by $hq$, one has
\begin{equation}
  \label{eq:11}
   \phi\dfrac{\p}{\p t}\left(\dfrac{1}{2}q^2\right) +
   \phi\ub\nabla\cdot\left(\dfrac{1}{2} q^2\right) = 0. 
 \end{equation}
 Multiplying $\eqref{eq:3}_1$ by $q^2/2$, one gets
\begin{equation}
  \label{eq:12}
   \dfrac{1}{2}q^2 \dfrac{\p}{\p t}\phi +
   \dfrac{1}{2} q^2\nabla\cdot(\phi\ub) = 0. 
 \end{equation}
Adding \eqref{eq:11} and \eqref{eq:12}, and integrating the resulting
equation over $\M$, one obtains
\begin{equation}
  \label{eq:13}
  \dfrac{d}{d t}\int_\M \dfrac{1}{2}\phi q^2 d{\mathbf x} = 0, 
\end{equation}
which proves that the potential enstrophy is conserved. In fact,
following the same procedure and still using only the first two
equations of \eqref{eq:3}, one can show that moments of
the PV $q$ of arbitrary orders are conserved, that is,
\begin{equation}
  \label{eq:14}
  \dfrac{d}{d t}\int_\M \dfrac{1}{2}\phi q^k d{\mathbf x} = 0, \qquad
  \forall\,\,k\ge 0. 
\end{equation}

It has been shown in \cite{Salmon1998-eg} that the inviscid shallow
water equations conform to the variational principles of the
Hamiltonian mechanics. Furthermore, the system can also be written as
the Hamiltonian canonical equations, and the energy conservation then
naturally follows from the skew-symmetry of the Poisson
bracket, and the potential enstrophy is only one of the countlessly
many singularities (or Casimirs) of the bracket, and thus is
conserved. Here, we present the Hamiltonian formulation of the SWEs,
which will serve as a tour-guide to the development of the numerical
schemes in the next section.

We define the Hamiltonian for the SWEs as
\begin{equation}
  \label{eq:15}
  H \equiv \int_\M \left(\dfrac{1}{2}\phi|\ub|^2+
    \dfrac{1}{2}g(\phi + b)^2\right) d{\mathbf x}.
\end{equation}
It is not a coincidence that the Hamiltonian is also the total energy
of the shallow water system (c.f.~\eqref{eq:10}). In this work, we use
$\delta$ to denote the variation of a functional. It is then easy to
see that, for the momentum variables,
\begin{equation}
  \label{eq:16}
  \delta H = \int_\M \left(\phi \ub\cdot \delta \ub +
    \left(\dfrac{1}{2}|\ub|^2 + g(\phi + b)\right)\delta\phi\right)
  d{\mathbf x}. 
\end{equation}
Thus, we obtain the following functional derivatives of the
Hamiltonian $H$ with respect to the thickness and the momentum,
\begin{subequations}\label{eq:17}
  \begin{align}
  \dfrac{\delta H}{\delta \phi} &= \dfrac{1}{2}|\ub|^2 + g(\phi+b)
                                  \equiv \Phi,\label{eq:17a}\\
  \dfrac{\delta H}{\delta \ub} &= \phi\ub.\label{eq:17b}
  \end{align}
\end{subequations}
We note that the right-hand side of \eqref{eq:17a}, usually called
the geopotential, also appears in the vector-invariant form
\eqref{eq:2} and the 
vorticity-divergence form \eqref{eq:3} of the shallow water
equations. This term will be denoted as $\Phi$ in the sequel.

Salmon has demonstrated in \cite{Salmon2004-tt, Salmon2005-pd,
  Salmon2007-dm} that it is advantageous to use the mass flux
$\phi\ub$ instead of the velocity $\ub$ itself as diagnostic
variables, because this choice allows one to write the kinetic energy,
which is a cubic polynomial of $\phi$ and $\ub$, in a binary form that
involves $\phi\ub$ and $\ub$,
\begin{equation}
  \label{eq:15a}
  H \equiv \int_\M \left(\dfrac{1}{2}\phi\ub\cdot\ub+
    \dfrac{1}{2}g(\phi + b)^2\right) d{\mathbf x}.
\end{equation}
% As will be clear momentarily,  $H$ can also be written in a binary 
% form that involves the 
% streamfunction, velocity potential, vorticity, and divergence.
When treated as a single variable, 
the mass flux $\phi\ub$  has a
Helmholtz decomposition (\cite{Girault1986-sr})
\begin{equation}
  \label{eq:18}
  \phi\ub = \nabla^\perp\psi + \nabla\chi,
\end{equation}
where $\psi$ and $\chi$ are the streamfunction and the velocity potential
respectively, and $\nabla^{\perp} = \kb\times\nabla$. In the classical Helmholtz decomposition, $\psi$ is
already assumed to satisfy the homogeneous Dirichlet boundary
condition (\cite{Girault1986-sr}); to enforce  no-flux boundary
condition \eqref{eq:5} on the flow, it only remains for the normal
derivative of the  velocity
potential $\chi$ to vanish on the boundary, and thus the boundary
conditions on $\psi$ and $\chi$ are
\begin{subequations}\label{eq:19}
  \begin{align}
    \psi &= 0, \qquad {\rm on}\; \p\M,\label{eq:19a}\\
    \dfrac{\p\chi}{\p n} &= 0,  \qquad{\rm on}\; \p\M.\label{eq:19b}
  \end{align}
\end{subequations}
The relation between $(\psi,\,\chi)$ and the vorticity and divergence
$(\zeta,\,\gamma)$ can be easily derived,
\begin{equation}
  \label{eq:24}
  \left\{
    \begin{aligned}
      \nabla\times\left(\phi^{-1}(\nabla^\perp\psi +
        \nabla\chi)\right) &= \zeta,\\
      \nabla\cdot\left(\phi^{-1}(\nabla^\perp\psi +
        \nabla\chi)\right) &= \gamma.
    \end{aligned}\right.
\end{equation}
Under the boundary conditions just given in \eqref{eq:19}and with a strictly positive
thickness field $\phi$, the system \eqref{eq:24} is a {\itshape coupled,
self-adjoint,} and 
{\itshape strictly elliptic} system for $(\psi,\,\chi)$.

In terms of the thickness, streamfunction and velocity potential, the
Hamiltonian \eqref{eq:15} can be expressed as
\begin{equation}
  \label{eq:20}
   H = \int_\M \left(\dfrac{1}{2} \phi^{-1}\left(
      |\nabla^\perp\psi|^2   + |\nabla\chi|^2 + 2\nabla^\perp\psi\cdot\nabla\chi\right)
     + \dfrac{1}{2}g(\phi + b)^2 \right)  d{\mathbf x}. 
 \end{equation}
% Substituting \eqref{eq:18} into \eqref{eq:15a}, and after some
% integration by parts, one rewrites the Hamiltonian in another binary
% form,
% \begin{equation}
%   \label{eq:15b}
%    H = \dfrac{1}{2}\int_\M \left( -\psi\zeta - \chi\gamma + g(\phi+b)^2\right)d{\mathbf x}.
%  \end{equation}

Substituting the expression \eqref{eq:18} into \eqref{eq:16}, we have
\begin{equation}
  \label{eq:21}
  \delta H = \int_\M \left(\nabla^\perp\psi \cdot \delta \ub +
    \nabla\chi\cdot \delta\ub + 
    \Phi \delta\phi\right)
  d{\mathbf x}. 
\end{equation}
The geopotential $\Phi$ has been defined in \eqref{eq:17a}. 
By integration by parts on the right-hand side,  one can transfer the differential
operators onto the velocity field to produce the vorticity and
divergence variables, and rewrite the equation as
\begin{equation}
  \label{eq:22}
  \delta H = \int_\M \left(-\psi \delta \zeta -
    \chi\delta\gamma + 
    \Phi \delta\phi\right)
  d{\mathbf x}. 
\end{equation}
The process goes through thanks to the boundary conditions
\eqref{eq:5} and \eqref{eq:19}, and the fact that differential
operators ($\nabla^\perp$ and $\nabla$) and the variation operator
$\delta$ are inter-changeable. 

From the relation \eqref{eq:22}, one easily derive the functional
derivatives of $H$ with respect to the thickness $\phi$, vorticity
$\zeta$ and divergence $\gamma$,
\begin{equation}\label{eq:23}
  \dfrac{\delta H}{\delta \phi} = \Phi,\qquad
  \dfrac{\delta H}{\delta \zeta} = -\psi,\qquad
  \dfrac{\delta H}{\delta \gamma} = -\chi.
\end{equation}

With the aid of the Poisson bracket \cite{Salmon2004-tt}, the system \eqref{eq:3} can be
put in the canonical form 
\begin{equation}
  \label{eq:25}
  \dfrac{\p}{\p t} F = \{F,\,H\}. 
\end{equation}
Here, $F$ represents a functional associated with the shallow water
system, and $H$ the Hamiltonian, which is also a functional. 
The Poisson bracket $\{\cdot,\,\cdot\}$ has three
components,
\begin{equation}
  \label{eq:26}
\{F,\,H\} = \{F,\,H\}_{\zeta\zeta} + \{F,\,H\}_{\gamma\gamma} +
\{F,\,H\}_{\phi\zeta\gamma},
\end{equation}
and each component is defined as follows,
\begin{subequations}\label{eq:27}
  \begin{align}
    \{F,\,H\}_{\zeta\zeta} =& \int_\M q J(F_\zeta,\,H_\zeta) d{\mathbf x},\label{eq:27a}\\
    \{F,\,H\}_{\gamma\gamma} =& \int_\M q J(F_\gamma,\,H_\gamma) d{\mathbf x},\label{eq:27b}\\
    \{F,\,H\}_{\phi\zeta\gamma}
     = &\int_\M \left[q(\nabla F_\gamma\cdot\nabla H_\zeta-\nabla
        H_\gamma \cdot\nabla F_\zeta )\right. +\label{eq:27c}\\
         &\qquad\left. (\nabla F_\gamma\cdot\nabla H_\phi-\nabla
       H_\gamma \cdot\nabla F_\phi )\right] d{\mathbf x}.\nonumber
  \end{align}
\end{subequations}
In the above, $F_\zeta$, etc., are  short-hands for the functional
derivatives $\delta F/\delta\zeta$, etc.
The potential vorticity $q$ for the SWEs is given by
\begin{equation}
  \label{eq:27-2}
  q = \dfrac{f+\zeta}{\phi},
\end{equation}
and $J(\cdot,\cdot)$ is the Jacobian operator   defined by
\begin{equation}
  \label{eq:27-1}
  J(a,b) = \nabla^\perp a \cdot\nabla b
\end{equation}
for any two functions $a$ and $b$.
The Jacobian operator is skew-symmetric w.r.t.~its two argument
functions.  

A state function, such as the vorticity $\zeta({\mathbf x},t)$, can be viewed as
a parametric functional, and can be cast into the functional form
using the parametric Kronner delta function, e.g.
\begin{equation*}
  F({\mathbf x},t) \equiv \zeta({\mathbf x},t) = \int_\M \delta({\mathbf x},{\mathbf y})\zeta({\mathbf x},t)d{\mathbf y},
  \qquad\forall\, {\mathbf x}\in\M, \, t > 0. 
\end{equation*}
By setting the functional $F$ to $\phi$, $\zeta$, and $\gamma$ in
\eqref{eq:26}, one recovers the evolution equation $\eqref{eq:3}_1$,
$\eqref{eq:3}_2$, and $\eqref{eq:3}_3$, respectively. Thus, all the
prognostic variables, the fluid thickness $\phi$, vorticity $\zeta$,
and divergence $\gamma$ evolve according to the differential equation
\eqref{eq:25}.

A crucial implication of the Hamiltonian approach is that 
an arbitrary functional of these
state variables, $F=F(\phi,\zeta,\gamma)$ also evolves according to
this equation.
\begin{theorem}\label{thm:evolv-analytic}
  An arbitrary functional $F=F(\phi,\,\zeta,\,\gamma)$ of the state
  variables also evolves according to the equation \eqref{eq:25}.
\end{theorem}

To see this, one takes the variation of this quantity,
and divides it by a variation in time $\delta t$ to obtain
\begin{align*}
  &{\delta F} = \int_\M \left( \dfrac{\delta F}{\delta
      \phi}\delta\phi + \dfrac{\delta F}{\delta
      \zeta}\delta{\zeta} + \dfrac{\delta F}{\delta
      \gamma}\delta{ \gamma}\right) d{\mathbf x},\\
  &\dfrac{d F}{d t} = \int_\M \left( \dfrac{\delta F}{\delta
      \phi}\dfrac{\p \phi}{\p t} + \dfrac{\delta F}{\delta
      \zeta}\dfrac{\p \zeta}{\p t} + \dfrac{\delta F}{\delta
      \gamma}\dfrac{\p \gamma}{\p t} \right) dt.
\end{align*}
Each state variable evolves according to equation \eqref{eq:25}. Thus
one can replace the time derivatives on the right-hand side by the
corresponding Poisson brackets,
\begin{equation*}
  \dfrac{d F}{d t} = \int_\M \left( \dfrac{\delta F}{\delta
      \phi}\left\{\phi,H\right\} + \dfrac{\delta F}{\delta
      \zeta}\{\zeta, H\} + \dfrac{\delta F}{\delta
      \gamma}\{\gamma, H\} \right) dt.
\end{equation*}
The next steps in this demonstration is to replace the Poisson brackets
by their explicit forms given in \eqref{eq:26} and \eqref{eq:27},
switch the order of integration, and show that the right-hand side of
the equation above equals $\{F,\,H\}$. Thus the equation \eqref{eq:25}
holds for an arbitrary functional $F$ of the state variables. 

It is clear that the Poisson bracket \eqref{eq:26} is skew-symmetric
with respect to its two arguments, and thus
\begin{equation*}
  \{H,\,H\} = 0,
\end{equation*}
and
\begin{equation}
  \label{eq:28}
  \dfrac{d}{d t} H = 0.
\end{equation}
The Hamiltonian for the shallow water system, which also identifies
with the total energy, is automatically conserved, thanks to the
skew-symmetry of the Poisson bracket.

Let $G(s)$ be an arbitrary function of its scalar argument. 
We now show that a quantity in the form of
\begin{equation}
  \label{eq:29}
  C = \int_\M\phi G(q) d{\mathbf x}
\end{equation}
is a Casimir of the Poisson bracket. We first notice that
\begin{equation}
  \label{eq:30}
  \delta C = \int_M (G(q) - qG'(q))\delta\phi + G'(q)\delta\zeta.  
\end{equation}
Thus, one has
\begin{equation}\label{eq:31}
  \dfrac{\delta C}{\delta \phi} = G(q) - qG'(q),\qquad
  \dfrac{\delta C}{\delta \zeta} = G'(q),\qquad
  \dfrac{\delta C}{\delta \gamma} = 0.
\end{equation}
Let $A$ be an arbitrary functional. After some algebraic
operations, one finds that 
\begin{equation}
  \label{eq:32}
  \{C,\,A\} = \int_\M \nabla C_\phi \cdot\nabla^\perp A_\zeta d{\mathbf x}. 
\end{equation}
As before,  $C_\phi$ and $A_\zeta$ are short-hands for the functional
derivatives. 
We assume that the quantity $A_\zeta$ satisfies the boundary
condition that
\begin{equation}
  \label{eq:33}
  \nabla^\perp A_\zeta \cdot\nb = 0,\qquad\p\M.
\end{equation}
Then we have
\begin{equation*}
  \{C,\,A\} = \int_{\p\M} C_\phi \nabla^\perp A_\zeta \cdot \nb ds =
  0.  
\end{equation*}
We thus have shown that
\begin{equation}
  \{C,\,A\} = 0, \qquad\forall\, A\,\, \textrm{satisfying
  }\eqref{eq:33}.\label{eq:34} 
\end{equation}
The Hamiltonian $H$
clearly satisfies the condition \eqref{eq:33}. As a direct consequence
of \eqref{eq:34}, an infinite number of quantities, in the form  of
$C$ of \eqref{eq:29}, are conserved under the shallow water
system. In particular, the mass (with $G(q)=1$), total circulation
(with $G(q) = q$), and potential enstrophy (with $G(q) = q^2/2$)
of the system are conserved.  

\section{Specification of the schemes}\label{sec:discrete}
The goal is to develop numerical schemes that conserve energy and/or
(potential) enstrophy over arbitrary bounded domains. Our approach
is to use the Hamiltonian formulation.
We use an unstructured and orthogonal dual mesh, in which the dual
cell edges are perpendicular to the associated primal cell edges. A
typical example is the Delaunay-Voronoi dual mesh. Detailed
specifications of the meshes are given in Appendix
\ref{s:mesh}. Specifications of discrete variables and discrete
differential operators on these meshes are listed in Appendix
\ref{s:ddo}. Finally, a discrete vector calculus, including many of
the identities and formulas that will be frequently used in the
sequel, is presented in Appendix \ref{sec:dvc}. 

The focus will be on the
symmetry of the Poisson brackets. Per discussions from the previous
section, (skew-)symmetries will naturally lead to conservative
properties that are desirable for numerical simulations of large-scale
geophysical flows.
The canonical equations for the fluids contain essentially two
ingredients: the Hamiltonian and the Poisson bracket. Thus, a
numerical discretization of the fluid system, based on the Hamiltonian
formulation, naturally comprises two phases: the Hamiltonian
phase approximating the Hamiltonian, and the Poisson phase
approximating the Poisson bracket. It will become apparent later that
these two phases are independent of each other.

\subsection{The Hamiltonian phase}
We enter the Poisson phase first. The results from this phase are
common to all the schemes that will be presented later.

As a vorticity-divergence based system, the prognostic variables are
 thickness $\phi_h$, vorticity $\zeta_h$, and divergence
$\gamma_h$ at the (primary) cell centers.  One will also need the
following diagnostic variables: 
the geopotential $\Phi_h$,
the streamfunction $\psi_h$, and the velocity potential $\chi_h$. The current
work adopts the convention that subscripted variables (by $_h$, $_i$,
etc.) are discrete, while un-subscripted variables are continuous;
un-accented discrete variables are defined at cell centers, while
$\hat{\hphantom{\psi}}$ on the top designates edge-defined variables,
and $\tilde{\hphantom{\psi}}$ designates vertex-defined variables. 

One approximates the Hamiltonian, defined in \eqref{eq:20}, by the
discrete Hamiltonian, still denoted as $H$,
\begin{multline}
  \label{eq:35}
  % H = \int_\M \left\{ \hat\phi_h^{-1}\left|\nabla_h^\perp\tilde\psi_h +
  %   \nabla_h\chi_h\right|^2 + \dfrac{1}{2}g(\phi_h + b_h)^2\right\} d{\mathbf x}.
  H = \int_\M \left\{ \hat\phi_h^{-1}\left( \left|\nabla_h^\perp\psi_h\right|^2 +
    \left|\nabla_h\chi_h\right|^2 +
    \nabla_h^\perp\tilde\psi_h\cdot\nabla\chi_h 
    +\nabla_h^\perp\psi_h\cdot\nabla_h\tilde\chi_h\right) \right. \\
 \left.\hphantom{\nabla_h^\perp\tilde\psi_h\cdot\nabla\chi_h}
  +\dfrac{1}{2}g(\phi_h + b_h)^2\right\} d{\mathbf x}. \quad
\end{multline}
The factor of $1/2$ in the kinetic energy in \eqref{eq:20} has
disappeared in the discrete version here due to the fact that only one
component is used in the inner product of the vector fields. One also
notes that the skew-symmetry of the Jacobian term is preserved in the
approximation here, which leads to a symmetric elliptic system later. 

We now derive formulas for the discrete geopotential, vorticity, and
divergence. These formulas will connect the diagnostic variables
$\psi_h$ and $\chi_h$ to the prognostic variables $\phi_h$, $\zeta_h$,
and $\gamma_h$ of the system. We begin by writing down the variation
of $H$ defined in \eqref{eq:35},
\begin{multline*}
  \delta H = \int_\M \left\{ -\hat\phi_h^{-2}\left(\left|\nabla_h^\perp
        \psi_h\right|^2 + \left|\nabla_h\chi_h\right|^2 +
      \nabla_h^\perp\tilde\psi_h\cdot\nabla_h \chi_h +  
    \nabla_h^\perp\psi_h\cdot\nabla_h\tilde\chi_h\right)
      \delta\hat\phi_h  \right.\\
    +\hat\phi^{-1}_h
    \left(2\nabla_h^\perp\psi_h\cdot\delta\nabla^\perp_h\psi_h +
      2\nabla_h\chi_h\cdot\delta\nabla_h\chi_h \right.\\
      +\nabla^\perp_h\tilde\psi_h\cdot\delta\nabla_h\chi_h +
      \nabla_h\chi_h\cdot\delta\nabla^\perp_h \tilde\psi_h + 
     \left. \sgradh\psi_h\cdot\delta\gradh\tilde\chi_h +
      \gradh\tilde\chi_h\cdot\delta\sgradh\psi_h \right)  \\
    \left.\hphantom{\left|\nabla_h^\perp\psi_h\right|^2}
      +g(\phi_h + b_h)\delta\phi_h\right\} d{\mathbf x}.\quad
\end{multline*}
The sign of the first term inside the integral is not right for the
geopotential. Switching the sign of the first term, and
combining the difference with the terms in the middle, one can write 
 $\delta H$ as
\begin{multline*}
  \delta H = \int_\M \left\{ \hat\phi_h^{-2}  \left(\left|\nabla_h^\perp
        \psi_h\right|^2 + \left|\nabla_h\chi_h\right|^2 +
      \nabla^\perp\tilde\psi_h\cdot\nabla\chi_h +  
    \nabla^\perp\psi_h\cdot\nabla_h\tilde\chi_h\right)
  \delta\hat\phi_h  {}\right.
\\+g(\phi_h  + b_h)\delta\phi_h  {} \\
  +2\sgradh\psi_h\cdot\delta\left(\hat\phi_h^{-1}\sgradh\psi_h\right) +
  2\gradh\chi_h\cdot\delta\left(\hat\phi_h^{-1}\gradh\chi_h\right) +
  \sgradh\tilde\psi_h\cdot\delta\left(\hat\phi_h^{-1}\gradh\chi_h\right)
   {}\\
 \left. +\gradh\chi_h\cdot\delta\left(\hat\phi_h^{-1}\sgradh\tilde\psi_h\right) +
 \sgradh\psi_h\cdot\delta\left(\hat\phi_h^{-1}\gradh\tilde\chi_h\right) +
  \gradh\tilde\chi_h\cdot\delta\left(\hat\phi_h^{-1}\sgradh\psi_h\right)
    \right\} d{\mathbf x}
\end{multline*}
% One then rewrite the integral above as the inner products of discrete
% fields,
% \begin{multline*}
%   \delta H = \left(\hat\phi_h^{-2} \left|\nabla_h^\perp
%     \tilde\psi_h + \nabla_h\chi_h\right|^2,\, \delta\hat\phi_h\right)
% + \left( g(\phi_h+ b_h),\,\delta\phi_h \right) + \\
%     2\left(\nabla_h^\perp\tilde\psi_h + \nabla_h\chi_h, -\hat\phi_h^{-2} \left(\nabla_h^\perp
%     \tilde\psi_h + \nabla_h\chi_h\right)\delta\hat\phi_h +\hat\phi^{-1}_h
%     \delta (\nabla_h^\perp\tilde\psi_h + \nabla_h\chi_h)\right).
% \end{multline*}
% The first two inner products can be combined together if the second
% components, i.e.~the differentials, to both of them can be made the
% same, which can be 
% accomplished 
With an application of the adjoint identity
\eqref{eq:vc1}, the terms in the first two lines can be combined
together. 
The discrete integration by parts formulas \eqref{eq:vc3},
\eqref{eq:vc4}, \eqref{eq:vc5}, and \eqref{eq:vc6}, and the adjoint
identity \eqref{eq:vc2} are applied to 
the remaining terms, transferring the differential operators and
the remapping operators from the first components of the inner products
to the second components. After regrouping these terms based on their
associations with the streamfunction $\psi_h$ or the velocity
potential $\chi_h$, one obtains
\begin{multline*}
  \delta H = \int_\M \left\{\left(
      \widehat{\hat\phi_h^{-2}\left(\left|\nabla_h^\perp 
            \psi_h\right|^2 + \left|\nabla_h\chi_h\right|^2 +
          \nabla^\perp\tilde\psi_h\cdot\nabla\chi_h +  
    \nabla^\perp\psi_h\cdot\nabla_h\tilde\chi_h\right)} \right.\right. \\
+g(\phi_h  + b_h)\delta\phi_h  {}\bigg)\delta\phi_h \\
{} -
\psi_h\delta\left(\nabla_h\times\left(\hat\phi_h^{-1}\sgradh\psi_h\right)
+
\dfrac{1}{2}\left(\tilde{\gradh\times\left(\hat\phi_h^{-1}\nabla_h\chi_h\right)}
    +
    \nabla_h\times\left(\hat\phi_h^{-1}\nabla_h\tilde\chi_h\right)\right)\right) \\
\left. {} -
\chi_h\delta\left(\nabla_h\cdot\left(\hat\phi_h^{-1}\gradh\chi_h\right)
+
\dfrac{1}{2}\left({\gradh\cdot\left(\hat\phi_h^{-1}\nabla_h^\perp\tilde\psi_h\right)} 
    +
   \tilde{\nabla_h\cdot\left(\hat\phi_h^{-1}\nabla_h^\perp\psi_h\right)}
 \right)\right) 
\right\}d{\mathbf x}
\end{multline*}
We set  
\begin{equation}
  \label{eq:36}\left\{
    \begin{aligned}
\Phi_h &=   \widehat{\hat\phi_h^{-2}\left(\left|\nabla_h^\perp 
            \psi_h\right|^2 + \left|\nabla_h\chi_h\right|^2 +
          \nabla^\perp\tilde\psi_h\cdot\nabla\chi_h +  
    \nabla^\perp\psi_h\cdot\nabla_h\tilde\chi_h\right)} + g(\phi_h+
b_h),\\
 \zeta_h &= \nabla_h\times\left(\hat\phi_h^{-1}\sgradh\psi_h\right)
 +
 \dfrac{1}{2}\left(\tilde{\gradh\times\left(\hat\phi_h^{-1}\nabla_h\chi_h\right)}
     +
     \nabla_h\times\left(\hat\phi_h^{-1}\nabla_h\tilde\chi_h\right)\right),\\
 \gamma_h &= 
 \dfrac{1}{2}\left({\gradh\cdot\left(\hat\phi_h^{-1}\nabla_h^\perp\tilde\psi_h\right)} 
     +
    \tilde{\nabla_h\cdot\left(\hat\phi_h^{-1}\nabla_h^\perp\psi_h\right)}
  \right) + \nabla_h\cdot\left(\hat\phi_h^{-1}\gradh\chi_h\right).
    \end{aligned}\right.
\end{equation}
Then the variation of the discrete Hamiltonian $H$ can be written
as
\begin{equation}
  \label{eq:37}
  \delta H = \int_\M\left(\Phi_h\delta\phi_h - \psi_h\delta\zeta_h -
    \chi_h\delta\gamma_h \right) d{\mathbf x}.
\end{equation}
This relation leads to the discrete analogues of the functional
derivatives given in \eqref{eq:23},
\begin{equation}
  \label{eq:38}
    \dfrac{\delta H}{\delta \phi_h} = \Phi_h,\qquad
  \dfrac{\delta H}{\delta \zeta_h} = -\psi_h,\qquad
  \dfrac{\delta H}{\delta \gamma_h} = -\chi_h.
\end{equation}

% Equation set made for the slides
%     \begin{equation*}
% \LEFTRIGHT\{.{
%     \begin{aligned}
% &\nabla_h\times\left(\hat\phi_h^{-1}\sgradh\psi_h\right)
%  +
%  \frac{1}{2}\left(\tilde{\gradh\times\left(\hat\phi_h^{-1}\nabla_h\chi_h\right)}
%      +
%      \nabla_h\times\left(\hat\phi_h^{-1}\nabla_h\tilde\chi_h\right)\right)
% &  &= & &\zeta_h, \\
%  &
%   \frac{1}{2}\left(\gradh\cdot\left(\hat\phi_h^{-1}\nabla_h^\perp\tilde\psi_h\right)
%       +
%       \tilde{\nabla_h\cdot\left(\hat\phi_h^{-1}\nabla_h^\perp\psi_h\right)}   \right)
%    + \nabla_h\cdot\left(\hat\phi_h^{-1}\gradh\chi_h\right) & &=
%    & &\gamma_h,\\
%    &    \psi_i = 0,\,\,\forall\, i\in\BC,\qquad
%       \chi_0 = 0. & & &&
% \end{aligned}
% }
% \end{equation*}

When the thickness $\phi_h$, the vorticity $\zeta_h$, and the divergence
$\gamma_h$ are known,  the streamfunction $\psi_h$ and the velocity potential $\chi_h$ can be
recovered from the coupled elliptic system (last two equations of
\eqref{eq:36} under proper boundary 
conditions). On a global sphere, no boundary conditions are needed, but
both the second and last equation of \eqref{eq:36} contains
redundancy, and the solutions are not unique, for any solution plus
some constants will still be a solution to the system. To ensure
uniqueness, one can replace one equation from each set with an
equation that sets $\psi_h$ or $\chi_h$ to a fixed value at a certain
grid point. Here, without loss of generality, we pick cell $i=0$,
and set
\begin{equation}
  \label{eq:36a}\left\{
    \begin{aligned}
&\psi_0 = 0,\\
&\chi_0 = 0.
    \end{aligned}\right.
\end{equation}
On a bounded domain, the discrete streamfunction satisfies the
homogeneous Dirichlet boundary conditions, while the homogeneous
Neumann boundary conditions for the velocity potential $\chi_h$ are
only implicitly enforced through the specification of the divergence
operator, and the velocity potential is set to zero at  cell $0$ to
ensure unique solvability of the system. Specifically, the boundary
conditions for the coupled 
elliptic system on a bounded domains are
\begin{equation}
  \label{eq:36aa}\left\{
    \begin{aligned}
&\psi_i = 0,\qquad i\in\BC,\\
&\chi_0 = 0.
    \end{aligned}\right.
\end{equation}

We note that the relations in \eqref{eq:36} governing the diagnostic
variables $(\psi_h,\,\chi_h)$ and the prognostic variables
$(\zeta_h,\,\gamma_h)$, given a discrete thickness field $\hat\phi_h$,
are solely determined by the choice of the discrete Hamiltonian
$H$. Other choices are certainly possible; but the particularly
choice given in \eqref{eq:35} avoids extra remapping between cell
centers and cell vertices, and therefore are more accurate. The
skew-symmetry-preserving approximation to the Jacobian term in
\eqref{eq:35} also 
ensures that the linear operator on the right-hand side of
$\eqref{eq:36}_{2,3}$ is self-adjoint.  
% has the advantage that, for a positive
% but arbitrary thickness field, the discrete
% operator on the right-hand side of the last two equations in
% \eqref{eq:36} is negative definite and self-adjoint, making the system
% suitable for efficient linear solvers such as the conjugate gradient
% method.
To see this, we denote the linear operator as $A_h$, given a
strictly positive but arbitrary thickness field $\hat \phi_h$, 
and let $(\psi_h^1,\,\chi_h^1)$ and $(\psi_h^2,\,\chi_h^2)$ be two sets
of diagnostic variables. Then, utilizing the duality identities
\eqref{eq:vc1}, \eqref{eq:vc2} about the remappings and the  discrete
integration by parts formulas \eqref{eq:vc3} and \eqref{eq:vc4}, we
find that 
\begin{align*}
  &\hspace{-0.5cm}\left( A_h \left(\psi^1_h,\,\chi^1_h\right)^T,\,\,
  \left(\psi^2_h,\,\chi^2_h\right)^T\right)\\
  \equiv
  &\left(\nabla_h\times\left(\hat\phi_h^{-1}\sgradh\psi^1_h\right) 
 + \dfrac{1}{2}\left(\tilde{
    \gradh\times\left(\hat\phi_h^{-1}\nabla_h\chi^1_h\right)}   
     + \nabla_h\times\left(\hat\phi_h^{-1}\nabla_h\tilde\chi^1_h
    \right)\right)  ,\, \psi^2_h \right) + \\
   & \left(\dfrac{1}{2}\left(
    {\gradh\cdot\left(\hat\phi_h^{-1}\nabla_h^\perp\tilde\psi^1_h\right)}  
     +
    \tilde{\nabla_h\cdot\left(\hat\phi_h^{-1}\nabla_h^\perp\psi^1_h\right)}
  \right) + \nabla_h\cdot\left(\hat\phi_h^{-1}\gradh\chi^1_h\right), \,\chi_h^2 \right) \\
    = &-2\left(\hat\phi^{-1}_h\sgradh\psi^1_h,\,\sgradh\psi^2_h\right)
        -
        \left(\hat\phi^{-1}_h\gradh\chi^1_h,\,\sgradh\tilde\psi^2_h\right)-
        \left(\hat\phi^{-1}_h\gradh\tilde\chi^1_h,\,\sgradh\psi^2_h\right)
  \\
 & -\left(\hat\phi^{-1}_h\sgradh\tilde\psi^1_h,\,\gradh\chi^2_h\right) -
  \left(\hat\phi^{-1}_h\sgradh\psi^1_h,\,\gradh\tilde\chi^2_h\right) -
  2\left(\hat\phi^{-1}_h\gradh\chi^1_h,\,\gradh\chi^2_h\right) \\
  =& 
\left(\psi^1_h,\, \nabla_h\times\left(\hat\phi_h^{-1}\sgradh\psi^2_h\right) 
 + \dfrac{1}{2}\left(\tilde{
    \gradh\times\left(\hat\phi_h^{-1}\nabla_h\chi^2_h\right)}   
     + \nabla_h\times\left(\hat\phi_h^{-1}\nabla_h\tilde\chi^2_h
    \right)\right)\right) + \\
   & \left(\chi^1_h,\, \dfrac{1}{2}\left(
    {\gradh\cdot\left(\hat\phi_h^{-1}\nabla_h^\perp\tilde\psi^2_h\right)}  
     +
    \tilde{\nabla_h\cdot\left(\hat\phi_h^{-1}\nabla_h^\perp\psi^2_h\right)}
  \right) + \nabla_h\cdot\left(\hat\phi_h^{-1}\gradh\chi^2_h\right)\right) \\
     \equiv&\left(\left(\psi^1_h,\,\chi^1_h\right)^T,\,\,
  A_h \left(\psi^2_h,\,\chi^2_h\right)^T\right).
\end{align*}
This shows that the operator $A_h$ is self-adjoint.
% Its negative
% definiteness comes from the fact that the above expression is strictly
% negative if $(\psi^1_h,\,\chi^1_h)$ and $(\psi^2_h,\,\chi^2_h)$ are
% identical and not equal to zero. 
The main disadvantage of this choice of discrete Hamiltonian is that
it is not known under what conditions the linear operator $A_h$ will
be (negative) definite. Under normal circumstances, e.g.~mild
variations in the thickness field, the strictly elliptic terms (second
term in $\eqref{eq:36}_{2,3}$) dominate, and the operator $A_h$ is
indeed negative definite. 
% special conditions, e.g.~a constant
% thickness field $\phi_h$ or a perfect rectangular or triangular mesh,
% the operator can be proven to be definite.
But a precise
characterization of the condition for the (negative) definiteness of
the operator $A_h$ is not known at this point. 

The next two sections deal with the Poisson phase, i.e.~discretization
of the Poisson bracket. Many approximations exist, each with
different properties. We shall present two approximations. The first
approximation preserves the skew-symmetry of the Poisson brackets, and
thus leading to a numerical scheme that attains the energy-conserving
property of the continuous 
system. However, the enstrophy is not conserved within this
scheme. The second approximation leads to a scheme that conserves both
energy and enstrophy.

\subsection{The Poisson phase and an energy conserving
  scheme}\label{sec:poisson1} 
In this subsection we present an approximation of the Poisson bracket
\eqref{eq:26} that is skew-symmetric to its two argument. We then
derive from this approximation an energy-conserving numerical scheme
for the shallow water equations.

The Poisson bracket \eqref{eq:26} is skew-symmetric w.r.t.~$F$ and
$H$, because each of its three components, defined in \eqref{eq:27},
is skew-symmetric. The first two components \eqref{eq:27a} and
\eqref{eq:27b} inherits its skew-symmetry from the Jacobian operator
$J(\cdot,\,\cdot)$. The third component \eqref{eq:27c} is
skew-symmetric due to the inherent permutation in $F$ and $H$ within
the specification. Thus,  the
approximations of $\{F,H\}_{\zeta\zeta}$ and $\{F,H\}_{\gamma\gamma}$
will be skew-symmetric if the approximation of the Jacobian operator
is so, while an approximation of
$\{F,H\}_{\phi\zeta\gamma}$ will be skew-symmetric if its retains the
inherent permutation of the corresponding discrete variables. 

The Jacobian, defined in \eqref{eq:27-1}, is an inner product of a
gradient vector and a skew-gradient vector. A direct approximation of
the inner product, e.g.~by the finite difference method, will
not be  skew-symmetric, because the 
approximation will inevitably involve discrete quantities defined at
different locations on the mesh.  However, we note that, due
to its skew-symmetry, the Jacobian can also be written as
\begin{equation}
  \label{eq:39}
  J(a,b) = \dfrac{1}{2}\nabla^\perp a \cdot\nabla b - \dfrac{1}{2}
  \nabla^\perp b \cdot\nabla a.
\end{equation}
This new expression for the Jacobian contains a permutation of its two
arguments. If this expression is adopted, then no matter how the inner
product of the skew-gradient and gradient vectors are
approximated, the approximation will be skew-symmetric as long as it
retains the permutation of the corresponding discrete variables. 
With $a_h$ and $b_h$ being the discretizations of the analytical
functions $a$ and $b$, respectively, we write down an approximation of
the Jacobian 
of \eqref{eq:39},
\begin{equation}
  \label{eq:40}
  J_h(a_h, b_h) \equiv \nabla_h^\perp \tilde a_h\cdot\nabla_h b_h -
  \nabla^\perp_h\tilde b_h \cdot \nabla_h a_h. 
\end{equation}
Here, $\tilde{\hphantom{a}}$ represents a mapping from cell centers to
cell vertices. Clearly, the discrete Jacobian $J_h(a_h, b_h)$ is
skew-symmetric with respect to its two arguments. This approach is
reminiscent of the discrete approximation of the trilinear form
$b(u,\,v,\,w)$ 
arising in the study of the Navier-Stokes equations (see
\cite{Temam2001-th}). 

With the Jacobian taken care of, the approximation of $\{F,
H\}_{\zeta\zeta}$ and $\{F,H\}_{\gamma\gamma}$ is now
straightforward,
\begin{align}
  &\{F,\,H\}_{\zeta\zeta} \approx \{F, \,H\}_{h,\zeta\zeta} \equiv
    \int_\M \hat q_h J_h\left(\dfrac{\delta F}{\delta \zeta_h},\,\dfrac{\delta
  H}{\delta \zeta_h}\right) d{\mathbf x},\label{eq:41}\\
  &\{F,\,H\}_{\gamma\gamma} \approx \{F, \,H\}_{h,\gamma\gamma}
    \equiv\int_\M \hat 
  q_h J_h\left(\dfrac{\delta F}{\delta \gamma_h},\,\dfrac{\delta
  H}{\delta \gamma_h}\right) d{\mathbf x}, \label{eq:42}
\end{align}
where $\hat q_h$ is the potential vorticity at cell edges, which is
a remapping of the PV $q_h$ at cell centers, and $q_h$ is defined
(see~\eqref{eq:27-2}) as
\begin{equation}
  \label{eq:27-3}
  q_h = \dfrac{f+\zeta_h}{\phi_h}.
\end{equation}

As pointed out above, the definition of $\{F,H\}_{\phi\zeta\gamma}$
already contains a permutation of its two argument, and this
permutation also appears in a straightforward approximation of the
expression by the finite difference method,
\begin{multline}
  \label{eq:43}
  \{F,\,H\}_{\phi\zeta\gamma} \approx \{F,\,H\}_{h,\phi\zeta\gamma}
  \equiv \\
  2\int_\M \hat q_h \left( \nabla_h
        F_{\gamma_h}\cdot\nabla_h H_{\zeta_h} - \nabla_h
        H_{\gamma_h}\cdot \nabla_h
        F_{\zeta_h}\right) d{\mathbf x} \\
    + 2\int_\M
    \left( \nabla_h
        F_{\gamma_h}\cdot\nabla_h H_{\phi_h} - \nabla_h
        H_{\gamma_h}\cdot \nabla_h
        F_{\phi_h} \right)
  d{\mathbf x}.
\end{multline}

The discrete Poisson bracket is the sum of the three expression given
above,
\begin{equation}
  \label{eq:44}
  \{F,\, H\}_{h} \equiv   \{F,\,
  H\}_{h,\phi\zeta\zeta} +   \{F,\,
  H\}_{h,\phi\gamma\gamma} +   \{F,\,
  H\}_{h,\phi\zeta\gamma}.
\end{equation}
We note that this discrete Poisson bracket is skew-symmetric,
and independent of the choice of the Hamiltonian $H$. A discrete scheme
based on this Poisson bracket is given by
\begin{equation}
  \label{eq:44a}
  \dfrac{\p F}{\p t} = \left\{F, \, H\right\}_h.
\end{equation}

We now derive the evolution equation for the discrete prognostic
variables $\phi_h$, $\zeta_h$, and $\gamma_h$. The
approximate Hamiltonian is given in \eqref{eq:35}, and its functional
derivatives w.r.t.~$\phi_h$, $\zeta_h$, and $\gamma_h$ have been given
in \eqref{eq:38}. These expressions will be used in the discrete
Poisson bracket. 

First, for an arbitrary cell index $i$, we take $F = \zeta_i$. $F$ can
be  
written as a functional of $\zeta_h$ with the help of a special kernel
function $\delta_{ih}$ as follows,
\begin{equation}
  \label{eq:45}
  F = \int_\M \delta_{ih} \zeta_h d{\mathbf x},
\end{equation}
where $\delta_{ih}$ is the discrete Kronecker delta function centered
on cell $i$, i.e.
\begin{equation}
  \label{eq:46}
  \delta_{ih} = \sum_j \delta_{ij}\chi_j,\qquad \delta_{ij} = \left\{
    \begin{array}[h]{cl}
      0 & \textrm{if } j \ne i,\\[3pt]
      \dfrac{1}{|A_i|} & \textrm{if } j=i.
    \end{array}\right.
\end{equation}
Thus, $F$ as a functional of $\zeta_h$ has derivatives
\begin{equation}
  \label{eq:47}
    \dfrac{\delta F}{\delta \phi_h} = 0,\qquad
  \dfrac{\delta F}{\delta \zeta_h} = \delta_{ih},\qquad
  \dfrac{\delta F}{\delta \gamma_h} = 0.
\end{equation}
Substituting the discrete functional derivatives, given in
\eqref{eq:38} and \eqref{eq:47}, into the sub-bracket $\{\zeta_i, 
H\}_{h,\zeta\zeta}$, one finds that
\begin{align*}
  \{\zeta_i,\,H\}_{h,\zeta\zeta}
  &= \int_\M \hat q_h
    \left(-\nabla^\perp_h\tilde{\delta_{ih}}\cdot
    \nabla_h\psi_h +
\nabla^\perp_h\tilde{\psi_{h}}\cdot
  \nabla_h\delta_{ih}\right) d{\mathbf x}\\
  &= \dfrac{1}{2}\int_\M \tilde\delta_{ih}\nabla_h\times\left(\hat q_h
    \nabla_h\psi_h\right)d{\mathbf x} - \dfrac{1}{2}\int_\M \delta_{ih}
    \nabla_h\cdot \left(\hat q_h 
    \nabla^\perp_h\tilde\psi_h\right)d{\mathbf x}.
\end{align*}
Both integration by parts in the above derivations go through for
arbitrary cell index $i$, thanks to the homogeneous boundary
conditions on $\psi_h$ and $\tilde\psi_h$. Applying the adjoint
identities \eqref{eq:vc2} to the above, we obtain
\begin{equation}\label{eq:48}
  \{\zeta_i,\,H\}_{h,\zeta\zeta}
  = \dfrac{1}{2}\left[\tilde{\nabla_h\times\left(\hat q_h
    \nabla_h\psi_h\right)}\right]_i - \dfrac{1}{2}\left[
    \nabla_h\cdot \left(\hat q_h 
    \nabla^\perp_h\tilde\psi_h\right) \right]_i.
\end{equation}
Since $F$ is independent of $\gamma_h$, the second component of the
Poisson bracket vanishes for this particular quantity, i.e.
\begin{equation}
  \label{eq:49}
  \{\zeta_i,\,H\}_{h,\gamma\gamma} = 0.
\end{equation}
Finally, substituting the discrete functional derivatives
\eqref{eq:38} and \eqref{eq:47} into the third component of the
Poisson bracket, we have
\begin{align*}
  \{\zeta_i, \, H\}_{h,\phi\zeta\gamma}
  &= 2\int_\M \hat q_h \nabla_h \delta_{ih}\cdot
   \nabla_h \chi_h d{\mathbf x} \\
  &= -\int_\M \delta_{ih}\nabla_h\cdot\left(\hat
    q_h\nabla_h\chi_h\right)d{\mathbf x}.
\end{align*}
% The integration by parts go through for all cell index $i$, thanks to
% the implied homogeneous Neumann boundary conditions on $\chi_h$. 
The integration-by-parts formula \eqref{eq:vc3} has been used.
The specification of the discrete Kronecker delta function
$\delta_{ih}$ allows us to write this component as
\begin{equation}
  \label{eq:50}
  \{\zeta_i,\,H\}_{h,\phi\zeta\gamma} = -\left[\nabla_h\cdot\left(\hat
    q_h\nabla_h\chi_h\right)\right]_i.
\end{equation}

For $F=\zeta_i$, the Poisson bracket $\{F,H\}_h$ is the sum of the
expressions \eqref{eq:48}--\eqref{eq:50}, and the semi-discrete
equation for $\zeta_i$ takes the form
\begin{equation}
  \label{eq:51}
  \dfrac{d}{dt}\zeta_i =
  \dfrac{1}{2}\left[\tilde{\nabla_h\times\left(\hat q_h 
    \nabla_h\psi_h\right)}\right]_i - \dfrac{1}{2}\left[
    \nabla_h\cdot \left(\hat q_h 
    \nabla^\perp_h\tilde\psi_h\right) \right]_i -
\left[\nabla_h\cdot\left(\hat 
    q_h\nabla_h\chi_h\right)\right]_i.
\end{equation}

The derivation of the evolution equation for the divergence $\gamma_i$
is made more complex by the non-vanishing boundary terms. A guiding
principle in handling the boundary terms is that the skew-symmetry of
the Poisson brackets should be preserved. We let
$F=\gamma_i$, for an arbitrary cell index $i\in\C$. Its functional
derivatives are found to be 
\begin{equation}
  \label{eq:52}
    \dfrac{\delta F}{\delta \phi_h} = 0,\qquad
  \dfrac{\delta F}{\delta \zeta_h} = 0,\qquad
  \dfrac{\delta F}{\delta \gamma_h} = \delta_{ih}.
\end{equation}

Since $F$ here is independent of the state variables $\zeta_h$, the
first component of the Poisson bracket vanishes, i.e.
\begin{equation}
  \{\gamma_i,\, H\}_{h,\zeta\zeta}
  = 0.\label{eq:53a}  
\end{equation}

Substituting the discrete functional derivative $H_{\gamma_h}$
and the discrete functional derivative  $F_{\gamma_h}= \delta_{ih}$
into the second components of the Poisson bracket, and
obtain
\begin{equation}\label{eq:53b}
  \{\gamma_i,\,H\}_{h,\gamma\gamma} =
  -\int_\M \hat q_h\nabla^\perp_h\tilde\delta_{ih} \cdot
  \nabla_h\chi_h d{\mathbf x} + \int_\M \hat
    q_h\nabla^\perp_h\tilde\chi_h\cdot\nabla_h \delta_{ih} d{\mathbf x}.   
\end{equation}
If $i\in\IC$, i.e.~$i$ refers to an interior cell, then
$\tilde\delta_{ih}$ vanishes on the boundary. We apply the integration-by-parts
formula \eqref{eq:vc4}  to the
first inner product above to obtain
\begin{equation*}
\int_\M \hat q_h\nabla^\perp_h\tilde\delta_{ih} \cdot
  \nabla_h\chi_h d{\mathbf x}  = -\dfrac{1}{2}
  \int_\M \tilde\delta_{ih} \nabla_h\times(\hat 
    q_h\nabla_h\chi_h)d{\mathbf x}. 
\end{equation*}
Applying the  adjoint identity \eqref{eq:vc2}  to the right-hand side,
and use the definition for the discrete delta function $\delta_{ih}$, 
we derive that
\begin{equation}\label{eq:53b3}
\int_\M \hat q_h\nabla^\perp_h\tilde\delta_{ih} \cdot
  \nabla_h\chi_h d{\mathbf x}
  = -\dfrac{1}{2}\left[\tilde{\nabla_h\times(\hat 
    q_h\nabla_h\chi_h)}\right]_i. 
\end{equation}

\begin{figure}[h]
  \centering
  \includegraphics[width=3in]{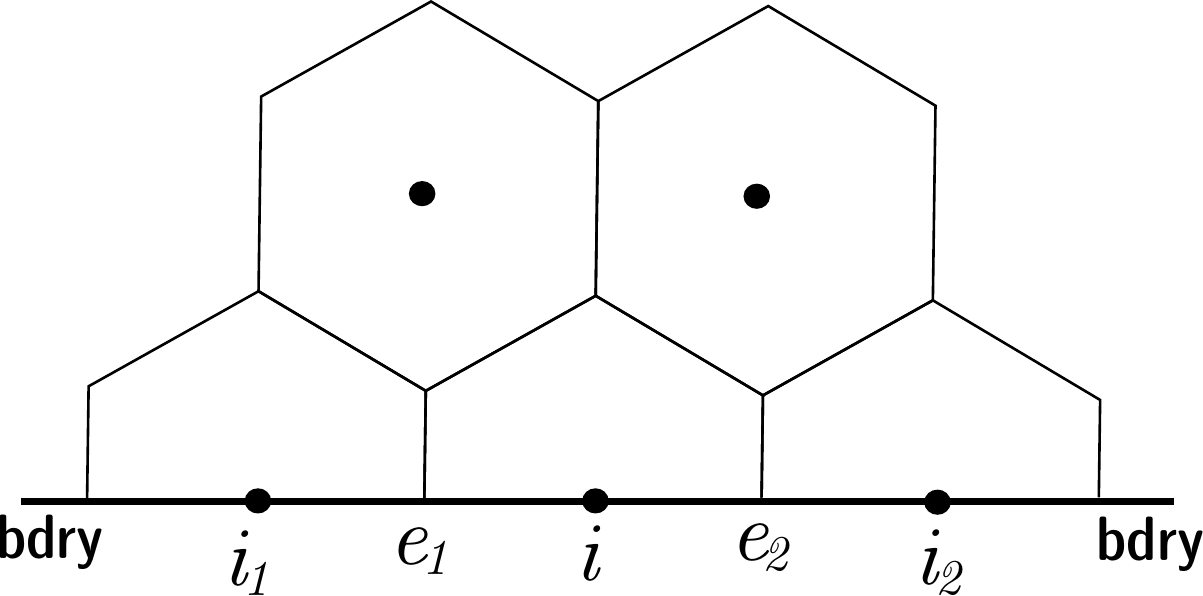}
  \caption{Diagram on boundary cells.}
  \label{fig:bdry}
\end{figure}

We now consider the boundary case when  $i\in\BC$. We let $i_1$ and
$i_2$ be two neighboring boundary cells, with $i_1,\,i,\,i_2$ ordered
in the counter clockwise direction. We designate the two boundary
edges of cell $i$ as $e_1$ and $e_2$, which are again ordered in the
counter clockwise direction (see Figure \ref{fig:bdry}).
Applying the discrete
integration by parts formula \eqref{eq:vc4} to the first inner product
on the 
right-hand side 
of \eqref{eq:53b} and using the fact that $\hat\delta_{ih}$ vanishes
everywhere on the boundary except on edges $e_1$ and $e_2$ around cell
$i$, we have
\begin{multline}\label{eq:53b6}
  \int_\M \hat q_h \nabla^\perp_h\tilde\delta_{ih}\cdot\nabla_h\chi_h d{\mathbf x}
  = -\dfrac{1}{2}
\int_\M \tilde\delta_{ih}\nabla_h\times(\hat 
    q_h\nabla_h\chi_h)d{\mathbf x} \\
  - \dfrac{1}{2}\left(\hat\delta_{i,e_1}
\hat q_{e_1} [\nabla_h\chi_h]_{e_1} d_{e_1}\sum_{\nu\in VE(e_1)}
t_{e_1,\nu} + \hat\delta_{i,e_1}
\hat q_{e_1} [\nabla_h\chi_h]_{e_1} d_{e_1}\sum_{\nu\in VE(e_1)}
t_{e_1,\nu} \right).
\end{multline}
We approximate the values of $\tilde\delta_{ih}$  at edge $e_1$ and
$e_2$ by 
\begin{equation}\label{eq:53b7}
  \hat\delta_{i,e_1} = \hat\delta_{i,e_2} \approx
  \dfrac{1}{2}\delta_{i,i} = \dfrac{1}{2|A_i|}.
\end{equation}
Here $|A_i|$ refers to the area of cell $i$.
We also note that, based on the assumed counter-clockwise orientation
of $\{e_1,\,e_2\}$ and $\{i_1,\,i,\,i_2\}$, we have
\begin{equation}
  \label{eq:53b8}
  \sum_{\nu\in VE(e_1)} t_{e_1,\nu} = n_{e_1,i},\qquad 
  \sum_{\nu\in VE(e_2)} t_{e_2,\nu} = -n_{e_2,i}.
\end{equation}
Substituting \eqref{eq:53b7} and \eqref{eq:53b8}, as well as the
specification \eqref{eq:gradc} for $[\nabla_h\chi_h]_{e_1}$ and
$[\nabla_h\chi_h]_{e_2}$ into \eqref{eq:53b6}, we obtain, for a
boundary cell $i$,
\begin{multline}\label{eq:53b1}
  \int_\M  \hat
  q_h \nabla^\perp_h\tilde\delta_{ih}\cdot\nabla_h\chi_h d{\mathbf x}
  = -\dfrac{1}{2}
  \left[\tilde{\nabla_h\times(\hat 
      q_h\nabla_h\chi_h)}\right]_i \\
+ \dfrac{1}{4|A_i|}\left(\hat q_{e_1} (\chi_i -\chi_{i_1}) + \hat
  q_{e_2} (\chi_{i_2} - \chi_i)\right). \qquad\quad
\end{multline}

For the second integral on the right-hand side of \eqref{eq:53b},
we apply the integration by parts formula \eqref{eq:vc3} directly,
\begin{equation}\label{eq:53b2}
\int_\M \hat
  q_h\nabla^\perp_h\tilde\chi_h\cdot\nabla_h \delta_{ih} d{\mathbf x}
  = -\dfrac{1}{2}\int_\M \nabla_h\cdot\left(\hat
  q_h\nabla^\perp_h\tilde\chi_h\right)\delta_{ih} d{\mathbf x} =
  -\dfrac{1}{2}\left[\nabla_h\cdot\left(\hat
  q_h\nabla^\perp_h\tilde\chi_h\right)\right]_i.
\end{equation}

Combining \eqref{eq:53b}-\eqref{eq:53b2}, we conclude that
\begin{equation}
  \label{eq:53b4}
  \left\{\gamma_i, \, H\right\}_{h,\gamma\gamma} = \left\{
    \begin{aligned}
      & \dfrac{1}{2}\left[\tilde{\nabla_h\times(\hat 
    q_h\nabla_h\chi_h)}\right]_i - \dfrac{1}{2}\left[\nabla_h\cdot\left(\hat
  q_h\nabla^\perp_h\tilde\chi_h\right)\right]_i, & &\textrm{if }i\in IC,\\
    %   & \dfrac{1}{2}\left[\tilde{\nabla_h\times(\hat 
    % q_h\nabla_h\chi_h)}\right]_i - \dfrac{1}{2}\left[\nabla_h\cdot\left(\hat
    %   q_h\nabla^\perp_h\tilde\chi_h\right)\right]_i, & &\\
    & \langle\textrm{same from above}\rangle {}-{} & &\\
     &\hphantom{2345} \dfrac{1}{4|A_i|}\left(\hat q_{e_1} (\chi_i -\chi_{i_1}) + \hat
  q_{e_2} (\chi_{i_2} - \chi_i)\right), & &\textrm{if } i\in\BC.
    \end{aligned}\right.
\end{equation}

Substituting the discrete functional derivatives \eqref{eq:38} of $H$
and the discrete functional derivatives \eqref{eq:52} of $\gamma_i$
into the third component of the Poisson bracket \eqref{eq:44}, one
obtains 
\begin{equation}\label{eq:53b5}
\{\gamma_i,\, H\}_{h,\phi\zeta\gamma}
  = -2\int_\M \hat q_h \nabla_h\delta_{ih}\cdot\nabla_h\psi_h d{\mathbf x} +
  2\int_\M \nabla_h\delta_{ih}\cdot\nabla_h\Phi_h d{\mathbf x}.
\end{equation}
Applying the integration by parts formula \eqref{eq:vc3} to both inner
products in the above, one reaches at
\begin{equation}
\{\gamma_i,\, H\}_{h,\phi\zeta\gamma}
  = \left[\nabla_h\cdot\left(\hat q_h \nabla_h\psi_h\right)\right]_i
    - \Delta_h \Phi_h. \label{eq:53c}
\end{equation}

The Poisson bracket for $F=\gamma_i$ is a sum of all three expressions
in \eqref{eq:53a}, \eqref{eq:53b4}, and \eqref{eq:53c}. For interior
cells ($i\in\IC$),  the semi-discrete evolution
equation for this quantity has the form
\begin{equation}
  \label{eq:54}
  \dfrac{d}{dt}\gamma_i = \left[\nabla_h\cdot\left(\hat 
    q_h\nabla_h\psi_h\right)\right]_i + 
  \dfrac{1}{2}\left[\tilde{\nabla_h\times\left(\hat q_h 
    \nabla_h\chi_h\right)}\right]_i - \dfrac{1}{2}\left[
    \nabla_h\cdot \left(\hat q_h 
    \nabla^\perp_h\tilde\chi_h\right) \right]_i - \left[\Delta_h
  \Phi_h\right]_i. 
\end{equation}
For boundary cells ($i\in\BC$),
\begin{multline}
  \label{eq:54a}
  \dfrac{d}{dt}\gamma_i = \left[\nabla_h\cdot\left(\hat 
    q_h\nabla_h\psi_h\right)\right]_i + 
  \dfrac{1}{2}\left[\tilde{\nabla_h\times\left(\hat q_h 
    \nabla_h\chi_h\right)}\right]_i - \dfrac{1}{2}\left[
    \nabla_h\cdot \left(\hat q_h 
    \nabla^\perp_h\tilde\chi_h\right) \right]_i \\
 - \left[\Delta_h
  \Phi_h\right]_i - \dfrac{1}{4|A_i|}\left(\hat q_{e_1} (\chi_i
  -\chi_{i_1}) + \hat  q_{e_2} (\chi_{i_2} - \chi_i)\right).  
\end{multline}

Finally, by taking $F=\phi_i$,  we have
\begin{equation}
  \label{eq:55}
    \dfrac{\delta F}{\delta \phi_h} = \delta_{ih},\qquad 
  \dfrac{\delta F}{\delta \zeta_h} = 0,\qquad
  \dfrac{\delta F}{\delta \gamma_h} = 0.
\end{equation}
The three components of the Poisson bracket for this quantity  can
be calculated as before,
\begin{align}
  &\{\phi_i,\, H\}_{h,\zeta\zeta}
  = 0,\label{eq:56a}\\
  &\{\phi_i,\, H\}_{h,\gamma\gamma}
  = 0,\label{eq:56b}\\ 
  &\{\phi_i,\, H\}_{h,\phi\zeta\gamma}
  =  - \Delta_h \chi_h. \label{eq:56c}
\end{align}

The Poisson bracket for $F=\phi_i$ is the sum of these expression, and
the semi-discrete evolution equation for it can be simply written as
\begin{equation}
  \label{eq:57}
  \dfrac{d}{dt} \phi_i = -\left[\Delta_h \chi_h\right]_i.
\end{equation}

For  reference, we summarize the full energy-conserving
semi-discrete scheme for the shallow water equations as follows,
\begin{equation}
  \label{eq:58}
  \left\{
    \begin{aligned}
      \dfrac{d}{dt} \phi_i =
      & -\left[\Delta_h \chi_h\right]_i,\\
      \dfrac{d}{dt}\zeta_i   =
  &\dfrac{1}{2}\left[\tilde{\nabla_h\times\left(\hat q_h 
    \nabla_h\psi_h\right)}\right]_i - \dfrac{1}{2}\left[
    \nabla_h\cdot \left(\hat q_h 
    \nabla^\perp_h\tilde\psi_h\right) \right]_i -
\left[\nabla_h\cdot\left(\hat 
    q_h\nabla_h\chi_h\right)\right]_i,\\
\dfrac{d}{dt}\gamma_i =
  &\dfrac{1}{2}\left[\tilde{\nabla_h\times\left(\hat q_h 
    \nabla_h\chi_h\right)}\right]_i - \dfrac{1}{2}\left[
    \nabla_h\cdot \left(\hat q_h 
    \nabla^\perp_h\tilde\chi_h\right) \right]_i +
\left[\nabla_h\cdot\left(\hat q_h\nabla_h\psi_h\right)\right]_i\\
&- \left[\Delta_h
  \Phi_h\right]_i - \dfrac{1}{4|A_i|}\left(\hat q_{e_1} (\chi_i
  -\chi_{i_1}) + \hat q_{e_2} (\chi_{i_2} - \chi_i)\right).
    \end{aligned}\right.
\end{equation}
The terms preceded by $1/4|A_i|$ in the equation for $\gamma_i$ only
appear for boundary cells ($i\in\BC$). 

A discrete analogue of Theorem \ref{thm:evolv-analytic} can also be
established. 
\begin{theorem}\label{thm:evolv}
  An arbitrary functional $F=F(\phi_h,\,\zeta_h,\,\gamma_h)$  of
  the discrete state variables  evolves according to the equation
  \eqref{eq:44a}. 
\end{theorem}
The proof of this theorem relies on the specific form of the discrete
Poisson bracket \eqref{eq:44}, and unlike for the analytical case,
the boundary terms will show up during the process of integration by
parts. %The details are given in Appendix \ref{sec:proofs}.

\begin{proof}[Proof of Theorem \ref{thm:evolv}]
We let $F=F(\phi_h,\,\zeta_h,\,\gamma_h)$ be an arbitrary functional
of the state variables, which does not explicitly depend on time
$t$. We proceed as in the analytical case, and find the variation of
this functional,
\begin{equation*}
  \delta F = \int_\M\left(\dfrac{\delta F}{\delta \phi_h}\delta\phi_h
    + \dfrac{\delta F}{\delta \zeta_h}\delta\zeta_h + \dfrac{\delta
      F}{\delta \gamma_h}\delta\gamma_h  \right)d{\mathbf x}. 
\end{equation*}
Dividing both sides by $\delta t$, and noticing that the state
variables $\phi_h$, $\zeta_h$, and $\gamma_h$ each satisfies the
evolution equation \eqref{eq:44a}, one obtains
\begin{equation}
  \label{eq:59}
  \dfrac{d F}{d t} = \int_\M\left(\dfrac{\delta F}{\delta
      \phi_h}\{\phi_h, H\}_h
    + \dfrac{\delta F}{\delta \zeta_h}\{\zeta_h,\,H\}_h +
    \dfrac{\delta F}{\delta \gamma_h}\{\gamma_h,\,H\}_h
  \right)d{\mathbf x}.  
\end{equation}

We now separately examine the terms involved in the above. We first
observe that $\phi_h$ is independent of the other two state variables,
and therefore
\begin{multline*}
  \int_\M \dfrac{\delta F}{\delta
      \phi_h}\{\phi_h, H\}_h d{\mathbf x} = \int_\M \dfrac{\delta F}{\delta
      \phi_h}\{\phi_h, H\}_{h,\phi\zeta\gamma} d{\mathbf x} = \\
    \sum_{i\in\C}
  \left[\dfrac{\delta F}{\delta \phi_h}\right]_i \int_\M
  (-2)\nabla_h\delta_{ih}\cdot\nabla_h\left(\dfrac{\delta H}{\delta
      \gamma_h}\right) d{\mathbf y} |A_i|.  
\end{multline*}
Integration by parts in ${\mathbf y}$, and then switching the summation with
the integration, one has
\begin{equation*}
  \int_\M \dfrac{\delta F}{\delta
      \phi_h}\{\phi_h, H\}_h d{\mathbf x} =  \int_\M
  \nabla_h\cdot\left(\nabla_h\left(\dfrac{\delta H}{\delta
      \gamma_h}\right)  \right) \left(\sum_{i\in\C}
  \left[\dfrac{\delta F}{\delta \phi_h}\right]_i \delta_{ih}({\mathbf y}) |A_i|\right) d{\mathbf y}. 
\end{equation*}
The Kronecker-delta function $\delta_{ih}$ with respect to index $i$
is defined in \eqref{eq:46}. The summation over $i$ yields a discrete
scalar field $\delta F/\delta\phi_h$ with the dependent variable
${\mathbf y}$. Thus, one has
\begin{equation*}
  \int_\M \dfrac{\delta F}{\delta
      \phi_h}\{\phi_h, H\}_h d{\mathbf x} =  \int_\M
  \nabla_h\cdot\left(\nabla_h \left(\dfrac{\delta H}{\delta
      \gamma_h}\right) \right)
  \dfrac{\delta F}{\delta \phi_h}({\mathbf y}) d{\mathbf y}. 
\end{equation*}
Integrating by parts again in ${\mathbf y}$, one finds that
\begin{equation}\label{eq:60}
  \int_\M \dfrac{\delta F}{\delta
      \phi_h}\{\phi_h, H\}_h d{\mathbf x} =  \int_\M -2
  \nabla_h\left( \dfrac{\delta F}{\delta \phi_h}({\mathbf y})\right)
  \cdot\nabla_h \left(\dfrac{\delta H}{\delta
      \gamma_h}\right) d{\mathbf y}. 
\end{equation}

Turning our attention to the second term on the right-hand side of
\eqref{eq:59}, we note that $\zeta_h$ is independent of $\gamma_h$ and
$\phi_h$, and therefore,
\begin{equation}
  \label{eq:61}
  \int_\M \dfrac{\delta F}{\delta \zeta_h}\{\zeta_h,\,H\}_hd{\mathbf x} =
  \int_\M \dfrac{\delta F}{\delta
    \zeta_h} \left(\{\zeta_h,\,H\}_{h,\zeta\zeta} +
  \{\zeta_h,\,H\}_{h,\phi\zeta\gamma}\right) d{\mathbf x}. 
\end{equation}
The treatment of each of the two terms on the right-hand side above
follows the same procedure just presented, involving integration by
parts and switching the integrations and summations. One has, for the
first term above,
\begin{multline}
  \label{eq:62}
    \int_\M \dfrac{\delta F}{\delta \zeta_h}\{\zeta_h,\,H\}_{h,\zeta\zeta} d{\mathbf x}
    = \\
    \int_\M \left(
      \hat q_h\nabla^\perp_h\left(\tilde{\dfrac{\delta
            F}{\delta\zeta_h}}\right) \cdot\nabla_h\left(\dfrac{\delta
          H}{\delta \zeta_h}\right) - \hat
      q_h\nabla^\perp_h\left(\tilde{\dfrac{\delta 
            H}{\delta\zeta_h}}\right) \cdot\nabla_h\left(\dfrac{\delta
          F}{\delta \zeta_h}\right) \right)d{\mathbf x}  \equiv \{F,
    H\}_{h,\zeta\zeta}. 
\end{multline}
For the second term, one has
\begin{equation}
  \label{eq:63}
  \int_\M \dfrac{\delta F}{\delta
    \zeta_h}\{\zeta_h,\,H\}_{h,\phi\zeta\gamma} d{\mathbf x}  = \int_\M (-2) \hat
      q_h\nabla_h\left({\dfrac{\delta F}{\delta\zeta_h}}\right)
      \cdot\nabla_h\left(\dfrac{\delta 
          H}{\delta \gamma_h}\right)d{\mathbf x}.
\end{equation}
Substituting \eqref{eq:62} and \eqref{eq:63} into \eqref{eq:61}, one
therefore has
\begin{equation}
  \label{eq:64}
  \int_\M \dfrac{\delta F}{\delta \zeta_h}\{\zeta_h,\,H\}_hd{\mathbf x} =
  \{F,\,H\}_{h,\zeta\zeta} + \int_\M (-2) \hat
      q_h\nabla_h\left({\dfrac{\delta F}{\delta\zeta_h}}\right)
      \cdot\nabla_h\left(\dfrac{\delta 
          H}{\delta \gamma_h}\right)d{\mathbf x}.
\end{equation}

The variable $\gamma_h$ is independent of $\zeta_h$ and $\phi_h$, and
therefore the third term on the right-hand side of \eqref{eq:59} is
split into two terms,
\begin{equation}
  \label{eq:65}
    \int_\M \dfrac{\delta F}{\delta \gamma_h}\{\gamma_h,\,H\}_h d{\mathbf x} =
    \int_\M \dfrac{\delta F}{\delta\gamma_h}\left(
      \{\gamma_h,\,H\}_{h,\gamma\gamma} +
      \{\gamma_h,\,H\}_{h,\phi\zeta\gamma}\right) d{\mathbf x}.   
\end{equation}
Calculation of the first term on the right-hand side is the most
complex, as it involves non-zero boundary terms. We expand this term
using the specification for the Poisson bracket component
$\{\cdot,\,\cdot\}_{h,\gamma\gamma}$ and by replacing integration in
${\mathbf y}$ with a discrete summation,
\begin{multline}\label{eq:66}
    \int_\M \dfrac{\delta
      F}{\delta\gamma_h}\{\gamma_h,\,H\}_{h,\gamma\gamma}  =\\
    \sum_{i\in\C} |A_i|
  \left[\dfrac{\delta F}{\delta \phi_h}\right]_i\left( \int_\M
  \hat q_h \nabla_h^\perp
  \tilde\delta_{ih}\cdot\nabla_h(-\chi_h)d{\mathbf y}  -  \int_\M \hat q_h \nabla_h^\perp
  \left(-\tilde\chi_h\right) \cdot\nabla_h\delta_{ih}d{\mathbf y} \right).
\end{multline}
As demonstrated in \eqref{eq:53b3} or \eqref{eq:53b1}, integration by
parts on the first integral has different outcomes, depending on
whether $i\in\IC$ 
(interior cells) or $i\in\BC$ (boundary cells). More specifically,
\begin{multline*}
\int_\M
  \hat q_h \nabla_h^\perp
  \tilde\delta_{ih}\cdot\nabla_h(-\chi_h)d{\mathbf y} = \dfrac{1}{2}\int_\M
  \delta_{ih} \tilde{\nabla_h\times\left(\hat q_h\nabla_h\chi_h\right)}
  d{\mathbf y} \\+ \left\{
    \begin{aligned}
      &0 & &\textrm{ if } i\in\IC,\\
      &\dfrac{-1}{4|A_i|}\left(\hat q_{e_1} (\chi_i -\chi_{i_1}) + \hat
  q_{e_2} (\chi_{i_2} - \chi_i)\right) & &\textrm{ if } i\in\BC. 
    \end{aligned}\right.
\end{multline*}
For the boundary case, the geometric relations between $i$, $i_1$,
$i_2$, $e_1$ and $e_2$ are as depicted in Figure \ref{fig:bdry}.

Therefore, the
summation involving the first integral on the right-hand side of
\eqref{eq:66} can be written as
\begin{multline}\label{eq:67}
      \sum_{i\in\C} |A_i|
  \left[\dfrac{\delta F}{\delta \phi_h}\right]_i \int_\M
 \hat q_h \nabla_h^\perp
 \tilde\delta_{ih}\cdot\nabla_h(-\chi_h)d{\mathbf y} = \\
 \sum_{i\in\C} |A_i|
  \left[\dfrac{\delta F}{\delta \phi_h}\right]_i \int_\M\dfrac{1}{2}
    \delta_{ih}\tilde{\nabla_h\times \left( \hat q_h
        \nabla_h\chi_h\right)} d{\mathbf y} -{} \\
    \sum_{i\in\BC}
  \left[\dfrac{\delta F}{\delta\gamma_h}\right]_i \dfrac{\hat q_{e_1} (\chi_i -\chi_{i_1}) + \hat
  q_{e_2} (\chi_{i_2} - \chi_i)}{4|A_i|}\cdot|A_i|.
\end{multline}
 One switches the summation
and the integration on the first term on the right-hand side, and then
applies the adjoint identity to obtain
\begin{multline}\label{eq:68}
      \sum_{i\in\C} |A_i|
  \left[\dfrac{\delta F}{\delta \phi_h}\right]_i \int_\M
 \hat q_h \nabla_h^\perp
 \tilde\delta_{ih}\cdot\nabla_h(-\chi_h)d{\mathbf y} = \\
\int_\M\dfrac{1}{2}\tilde{\left[\dfrac{\delta F}{\delta \phi_h}\right]}
{\nabla_h\times \left( \hat q_h
        \nabla_h\chi_h\right)} d{\mathbf y} -{} \\
    \dfrac{1}{4}\sum_{i\in\BC}
  \left[\dfrac{\delta F}{\delta\gamma_h}\right]_i\left( {\hat q_{e_1}
      (\chi_i -\chi_{i_1}) + \hat 
  q_{e_2} (\chi_{i_2} - \chi_i)}\right).
\end{multline}
For the first term on the right-hand side, one applies the discrete
integration \eqref{eq:vc4}, and obtains
\begin{multline*}
\int_\M\dfrac{1}{2}\tilde{\left[\dfrac{\delta F}{\delta \phi_h}\right]}
{\nabla_h\times \left( \hat q_h
  \nabla_h\chi_h\right)} d{\mathbf y}
= -\int_\M\hat q_h \nabla_h^\perp \left(\tilde{\dfrac{\delta F}{\delta
      \gamma_h}}\right)\cdot \nabla_h \chi_h d{\mathbf y} \\
 { }- \dfrac{1}{2}\sum_{e\in\BE}
    \left[\hat{\dfrac{\delta F}{\delta\gamma_h}}\right]_e \hat q_e
    \left[\nabla_h\chi_h\right]_e  d_e \sum_{\nu\in\VE(e)} t_{e,\nu}.
  \end{multline*}
  The quantity $[\tilde{{\delta F}/{\delta
    \gamma_h}}]_e$ on the edge is computed using the formula
\eqref{eq:ce1}, and the discrete gradient
$[\nabla_h\chi_h]_e$ on the edge is computed using the
formula \eqref{eq:gradc}. After substituting these definitions into
the summation above on the boundary,
one obtains
\begin{multline*}
\int_\M\dfrac{1}{2}\tilde{\left[\dfrac{\delta F}{\delta \phi_h}\right]}
{\nabla_h\times \left( \hat q_h
  \nabla_h\chi_h\right)} dy
= 
-\int_\M\hat q_h \nabla_h^\perp \left(\tilde{\dfrac{\delta F}{\delta
      \gamma_h}}\right)\cdot \nabla_h \chi_h d{\mathbf y} \\
{ }+ \dfrac{1}{4}\sum_{e\in\BE}
     \left(\sum_{i\in\CE(e)} \left[\dfrac{\delta
             F}{\delta\gamma_h}\right]_i\right) \hat q_e 
     \left(\sum_{i\in\CE(e)} \chi_i n_{e,i}  \sum_{\nu\in\VE(e)} t_{e,\nu}\right).
  \end{multline*}
  Switching the two outermost summations, one have
\begin{multline*}
\int_\M\dfrac{1}{2}\tilde{\left[\dfrac{\delta F}{\delta \phi_h}\right]}
{\nabla_h\times \left( \hat q_h
  \nabla_h\chi_h\right)} d{\mathbf y}
= 
-\int_\M\hat q_h \nabla_h^\perp \left(\tilde{\dfrac{\delta F}{\delta
      \gamma_h}}\right)\cdot \nabla_h \chi_h d{\mathbf y} \\
{ }+ \dfrac{1}{4}
     \sum_{i\in\BC} \left[\dfrac{\delta
             F}{\delta\gamma_h}\right]_i \sum_{e\in\BE\cap\EC(i)} \hat q_e 
     \left(\sum_{i\in\CE(e)} \chi_i n_{e,i}  \sum_{\nu\in\VE} t_{e,\nu}\right).
   \end{multline*}
   One notes that the term in the parentheses represent a difference
   formula in the positive orientation, i.e.~counter-clockwise
   direction, on the boundary. Thus, with boundary cells
   $(i_1,\,i,\, i_2)$ ordered in the counter-clockwise direction,
   one can write the terms as
\begin{multline*}
\int_\M\dfrac{1}{2}\tilde{\left[\dfrac{\delta F}{\delta \phi_h}\right]}
{\nabla_h\times \left( \hat q_h
  \nabla_h\chi_h\right)} d{\mathbf y}
= 
-\int_\M\hat q_h \nabla_h^\perp \left(\tilde{\dfrac{\delta F}{\delta
      \gamma_h}}\right)\cdot \nabla_h \chi_h d{\mathbf y} \\
{ }+ \dfrac{1}{4}
     \sum_{i\in\BC}\left[\dfrac{\delta F}{\delta \gamma_h}\right]_i
     \left(\hat q_{e_1}(\chi_i - \chi_{i_1}) + \hat 
       q_{e_2} (\chi_{i_2} - \chi_i)\right).
   \end{multline*}
   One notices that the boundary terms exactly match the boundary terms
   in \eqref{eq:68}. Therefore, substituting the expression on the
   right-hand side into \eqref{eq:68}, one greatly simplifies it into 
\begin{multline}\label{eq:69}
      \sum_{i\in\C} |A_i|
  \left[\dfrac{\delta F}{\delta \phi_h}\right]_i \int_\M
 \hat q_h \nabla_h^\perp
 \tilde\delta_{ih}\cdot\nabla_h(-\chi_h)d{\mathbf y} =
-\int_\M\hat q_h \nabla_h^\perp \left(\tilde{\dfrac{\delta F}{\delta
      \gamma_h}}\right)\cdot \nabla_h \chi_h d{\mathbf y}. 
\end{multline}

The summation on the right-hand side of \eqref{eq:66} can be handled
in the usual way, with no complications from the boundary terms, and
one has
\begin{equation}
  \label{eq:70}
      \sum_{i\in\C} |A_i|
  \left[\dfrac{\delta F}{\delta \phi_h}\right]_i\left( -  \int_\M \hat q_h \nabla_h^\perp
  \left(-\tilde\chi_h\right) \cdot\nabla_h\delta_{ih}d{\mathbf y} \right) =
\int_\M \hat q_h
\nabla_h^\perp\tilde\chi_h\cdot\nabla_h\left(\dfrac{\delta F}{\delta
    \gamma_h}\right) d{\mathbf y}.
\end{equation}

With \eqref{eq:69} and \eqref{eq:70}, one therefore concludes from
\eqref{eq:66} that
\begin{equation}
  \label{eq:71}
    \int_\M \dfrac{\delta
      F}{\delta\gamma_h}\{\gamma_h,\,H\}_{h,\gamma\gamma}  = \{F,\, H\}_{h,\gamma\gamma}.
\end{equation}

The treatment of the second term on the right-hand side of
\eqref{eq:65} is straightforward without complications from the
boundary terms. One starts by replacing
$\{\gamma_h,\,H\}_{h,\phi\zeta\gamma}$ by its definition
\eqref{eq:43}, and noticing that $\gamma_h$ is independent of $\phi_h$
and $\zeta_h$, one has
\begin{multline*}
      \int_\M \dfrac{\delta F}{\delta\gamma_h}
      \{\gamma_h,\,H\}_{h,\phi\zeta\gamma} d{\mathbf x} = { }\\
      \sum_{i\in\C} |A_i|
      \left[\dfrac{\delta F}{\delta\gamma_h}\right]_i \int_\M
      \left(2\hat q_h \nabla_h\delta_{ih}\cdot\nabla_h \dfrac{\delta
          H}{\delta\zeta_h} +
        2\nabla\delta_{ih}\nabla_h\cdot\Phi_h\right) d{\mathbf y}. 
    \end{multline*}
    For both integrals, the discrete integration by parts formula
    \eqref{eq:vc3} applies without generating any boundary
    terms. Afterwards, one proceeds by switching the summation and the
    integration, and applying the discrete integration by parts
    formula again, then one has in the end
\begin{equation}\label{eq:72}
      \int_\M \dfrac{\delta F}{\delta\gamma_h}
      \{\gamma_h,\,H\}_{h,\phi\zeta\gamma} d{\mathbf x} = 
       \int_\M
      \left(2\hat q_h\nabla_h \dfrac{\delta F}{\delta\gamma_h}\cdot \nabla_h \dfrac{\delta
          H}{\delta\zeta_h} +
        2 \nabla_h \dfrac{\delta F}{\delta\gamma_h}
        \cdot\nabla_h\dfrac{\delta\Phi_h}{\delta\phi_h} \right) d{\mathbf y}. 
\end{equation}

From \eqref{eq:65}, \eqref{eq:71}, and \eqref{eq:72}, one concludes
that
\begin{multline}
  \label{eq:73}
      \int_\M \dfrac{\delta F}{\delta\gamma_h}
      \{\gamma_h,\,H\}_{h} d{\mathbf x} =
      \{F,\,H\}_{h,\gamma\gamma} + {}\\
       \int_\M
      \left(2\hat q_h\nabla_h \dfrac{\delta F}{\delta\gamma_h}\cdot \nabla_h \dfrac{\delta
          H}{\delta\zeta_h} +
        2 \nabla_h \dfrac{\delta F}{\delta\gamma_h}
        \cdot\nabla_h\dfrac{\delta H}{\delta\phi_h} \right) d{\mathbf y}. 
\end{multline}
Taking \eqref{eq:60}, \eqref{eq:64}, and \eqref{eq:73} into the
evolutionary equation \eqref{eq:59}, one has
\begin{multline*}
  \dfrac{d F}{d t} = \{F,\,H\}_{h,\zeta\zeta} +
  \{F,\,H\}_{h,\gamma\gamma} + { }\\
  \int_\M
      \left(2\hat q_h\nabla_h \left(\dfrac{\delta F}{\delta\gamma_h}\cdot \nabla_h \dfrac{\delta
          H}{\delta\zeta_h} - \dfrac{\delta H}{\delta\gamma_h}\cdot \nabla_h \dfrac{\delta
          F}{\delta\zeta_h}\right)\right.\\
    { }+
        \left. 2 \left(\nabla_h \dfrac{\delta F}{\delta\gamma_h}
        \cdot\nabla_h\dfrac{\delta H}{\delta\phi_h} - \nabla_h \dfrac{\delta H}{\delta\gamma_h}
        \cdot\nabla_h\dfrac{\delta F}{\delta\phi_h} \right)\right) d{\mathbf y}. 
\end{multline*}
One notes that the integral terms above constitutes the definition of
the component $\{\cdot,\,\cdot\}_{h,\phi\zeta\gamma}$ of the Poisson
bracket, which together with the other two components already present
in the equation, constitutes the definition of the Poisson bracket
itself, and therefore one finally has
\begin{equation}
  \label{eq:74}
  \dfrac{d F}{d t} = \{F,\,H\}_h.
\end{equation}
This equation holds for an arbitrary functional of the discrete
variables $\phi_h$, $\zeta_h$, and $\gamma_h$. 
\end{proof}

The result that the scheme \eqref{eq:58} conserves the total energy,
which is a functional of the discrete state variables, is then a
simple consequence of the skew-symmetry of the discrete Poisson bracket
$\{\cdot,\,\cdot\}_h$.

We now examine the conservation of quantities in the form of
\begin{equation}\label{eq:75}
  C = \int_\M \phi_h G(q_h) d{\mathbf x}. 
\end{equation}

When $G(q_h) = 1$, $C=\int_\M \phi_h d{\mathbf x}$ is the mass. Concerning its
functional derivatives, one finds
\begin{equation*}
  \dfrac{\delta C}{\delta \phi_h} = 1,\qquad
  \dfrac{\delta C}{\delta \zeta_h} = 0,\qquad
  \dfrac{\delta C}{\delta \gamma_h} = 0.
\end{equation*}
Substituting $C$ as $F$ into the Poisson bracket \eqref{eq:44}, one
has
\begin{equation*}
  \{C,\,H\}_h = \int_\M -2 \nabla_h C_{\phi_h}\cdot\nabla_h H_{\gamma_h}
  d{\mathbf x} = 0.
\end{equation*}
Hence, the mass is conserved in the scheme \eqref{eq:58}. 

When $G(q_h) = q_h$, $C = \int_\M \phi_h q_hd{\mathbf x} = \int_\M
(f+\zeta_h)d{\mathbf x}$ is the circulation. It has functional derivatives
\begin{equation*}
  \dfrac{\delta C}{\delta \phi_h} = 0,\qquad
  \dfrac{\delta C}{\delta \zeta_h} = 1,\qquad
  \dfrac{\delta C}{\delta \gamma_h} = 0.
\end{equation*}
Substituting these functional derivatives into the Poisson bracket
\eqref{eq:44}, one has
\begin{equation*}
  \{F,\,H\}_h = \int_\M \hat q_h\left(\nabla_h^\perp
    \tilde{C_{\zeta_h}}\cdot\nabla_h H_{\zeta_h} - 
  \nabla^\perp_h\tilde{H_{\zeta_h}}\cdot\nabla_h C_{\zeta_h} -
  2\nabla_h C_{\zeta_h} \cdot\nabla_h H_{\gamma_h} \right)d{\mathbf x} = 0.
\end{equation*}
Hence the total circulation is also conserved.

When $G(q_h) = q_h^2/2$, $C = \displaystyle \int_\M \phi_h q^2_h /2 d{\mathbf x}$ is the discrete
potential enstrophy. We note that the PV $q_h$ is given by
$q_h = (f+\zeta_h)/\phi_h$. Thus, the discrete potential enstrophy also takes
the form $C = \int_\M \phi_h^{-1} (f+\zeta_h)^2 d{\mathbf x}$. It has functional
derivatives
\begin{equation*}
  \dfrac{\delta C}{\delta \phi_h} = -\dfrac{1}{2}q^2_h,\qquad
  \dfrac{\delta C}{\delta \zeta_h} = q_h,\qquad
  \dfrac{\delta C}{\delta \gamma_h} = 0.
\end{equation*}
Due to the independence of the quantity $C$ from the divergence
$\gamma_h$, its Poisson bracket has only two components that might be
non-zero,
\begin{equation}\label{eq:76}
  \{C,\, H\}_h = \{C,\, H\}_{h,\zeta\zeta} + \{C,\,
  H\}_{h,\phi\zeta\gamma}. 
\end{equation}
Substituting the functional derivatives of $C$ and $H$ into the first
component, one has
\begin{equation*}
  \{C,\,H\}_{h,\zeta\zeta} = \int_\M \hat q_h\left(\nabla^\perp_h
    \tilde q_h \cdot\nabla_h H_{\zeta_h} - \nabla^\perp_h
    \tilde{H_{\zeta_h}} \cdot\nabla_h q_h \right) d{\mathbf x}.
\end{equation*}
With the mapping $\hat{\hphantom{q}}$ defined as in \eqref{eq:ce1}, the second
term in the integral above vanishes, for
\begin{equation*}
\int_\M \hat q_h\nabla^\perp_h
    \tilde{H_{\zeta_h}} \cdot\nabla_h q_h  d{\mathbf x} = \int_\M \nabla^\perp_h
    \tilde{H_{\zeta_h}} \cdot\nabla_h\left(\dfrac{1}{2} q^2_h\right)
    d{\mathbf x} = 0,   
\end{equation*}
by the virtue of Lemma \ref{lem:irrot}. Thus, one has
\begin{equation}\label{eq:77}
  \{C,\,H\}_{h,\zeta\zeta}  =\int_\M \hat q_h\nabla^\perp_h
    \tilde q_h \cdot\nabla_h H_{\zeta_h}.   
\end{equation}
One notes that one can modify the specification \eqref{eq:ce1} of the
mapping 
$\hat{\hphantom{q}}$ to make this remaining term vanishes. But then
the other term will not. There is no known method to make both terms
go away, which is just one sign of the limitations of a discrete
system with a finite number of degrees of freedom. 

Concerning the second term on the right-hand side of \eqref{eq:76},
one has, after substituting in the functional derivatives of $C$,
\begin{equation*}
  \{C,\,H\}_{h,\phi\zeta\gamma} = -2 \int_\M \left(\hat q_h \nabla_h q_h
  \cdot\nabla_h H_{\gamma_h} + \nabla_h\left(-\dfrac{1}{2}
    q^2_h\right) \cdot\nabla_h H_{\gamma_h} \right)d{\mathbf x}. 
\end{equation*}
Again, thanks to the specification \eqref{eq:ce1} of the mapping
$\hat{\hphantom{q}}$, one can bring $\hat q_h$ inside the gradient
operator, and arrives at
\begin{equation}\label{eq:78}
  \{C,\,H\}_{h,\phi\zeta\gamma} = -2 \int_\M \left(\nabla_h
    \left(\dfrac{1}{2}q^2_h\right) 
  \cdot\nabla_h H_{\gamma_h}  - \nabla_h\left(\dfrac{1}{2}
    q^2_h\right) \cdot\nabla_h H_{\gamma_h} \right)d{\mathbf x} = 0. 
\end{equation}
Combining \eqref{eq:77} and \eqref{eq:78} with \eqref{eq:76}, one
concludes that
\begin{equation}\label{eq:79}
  \{C,\, H\}_h =  \int_\M \hat q_h\nabla^\perp_h
    \tilde q_h \cdot\nabla_h H_{\zeta_h},
\end{equation}
which means that the potential enstrophy is not conserved under the scheme
\eqref{eq:58}. 

\subsection{An energy and enstrophy conserving scheme}
As just pointed out, the scheme \eqref{eq:58} fails to conserve the
enstrophy. The Nambu bracket
(\cite{Nambu1973-cm})
expands the Poisson bracket to include the potential enstrophy as a
third argument. Then, the conservation of potential enstrophy is
simply a consequence of the skew-symmetry of the expanded
bracket. % This suggests that, in order to obtain conservation of
% potential enstrophy, one should work with the Nambu bracket.
However,
as already pointed out by Salmon (\cite{Salmon2009-xr}), it is, in
general, 
difficult to incorporate boundary conditions into the Nambu bracket. 

We note that the only reason that the scheme \eqref{eq:58} fails to
conserve 
the potential enstrophy is  that the first component
$\{\cdot,\,\,\cdot\}_{h,\zeta\zeta}$ of the discrete Poisson bracket
does not vanish when the functional $F$ is taken as the discrete
potential enstrophy, unlike in the analytical case. Thus, we propose
to replace this component by a Nambu-style trilinear bracket, containing
the potential enstrophy as a third functional argument, which vanishes
when the functional $F$ is taken as the potential enstrophy. The
issue concerning boundary conditions is not a problem here due to the
fact that the streamfunction $\psi$ vanishes along the
boundary (homogeneous Dirichlet).

We shall first consider the analytical case, and then we deal with its
discretization. 
We let
\begin{equation}\label{eq:79a}
  Z = \int_\M \dfrac{1}{2}\phi q^2 d{\mathbf x}. 
\end{equation}
As we have seen, this quantity has functional derivatives
\begin{equation*}
  \dfrac{\delta Z}{\delta \phi_h} = -\dfrac{1}{2} q^2,\qquad
  \dfrac{\delta Z}{\delta \zeta_h} = q,\qquad
  \dfrac{\delta Z}{\delta \gamma_h} = 0.
\end{equation*}
Using this new variable, the component $\{F,\,H\}_{\zeta\zeta}$ of the
Poisson bracket  can be written as
\begin{equation*}
  \{F,\,H\}_{\zeta\zeta} = \int_\M Z_\zeta\nabla^\perp
  F_\zeta \cdot\nabla H_\zeta d{\mathbf x}.
\end{equation*}
We define the right-hand side as a trilinear bracket,
\begin{equation}
  \label{eq:80}
  \{F,\,H,\,Z\}_{\zeta\zeta\zeta} = \int_\M Z_\zeta\nabla^\perp
  F_\zeta \cdot\nabla H_\zeta d{\mathbf x}.
\end{equation}
Through a simple exercise and thanks to the homogeneous Dirichlet
boundary conditions on the streamfunction $\psi = - H_{\zeta_h}$, one
can show that the trilinear bracket \eqref{eq:80} is
skew-symmetric with respect to any two of the arguments, i.e.,
\begin{equation}
  \label{eq:81}
  \{F,\,H,\,Z\}_{\zeta\zeta\zeta} = -\{Z,\,H,\,F\}_{\zeta\zeta\zeta} =
- \{F,\,Z,\,H\}_{\zeta\zeta\zeta}. 
\end{equation}
That the Poisson bracket component $\{F,\,H\}_{\zeta\zeta}$ vanishes,
when $F$ takes the value of the potential enstrophy $Z$, is a direct
consequence of the skew-symmetry of the trilinear bracket. 
Also due to the skew-symmetry, the trilinear bracket is invariant
under even permutations, i.e.,
\begin{equation}
  \label{eq:82}
  \{F,\,H,\,Z\}_{\zeta\zeta\zeta} = \{Z,\,F,\,H\}_{\zeta\zeta\zeta} =
 \{H,\,Z,\,F\}_{\zeta\zeta\zeta}. 
\end{equation}

One notes that it is also possible to expand the other two components
of the Poisson bracket into trilinear brackets, but it becomes
problematic to establish the skew-symmetry and the invariance under
even permutations for them, due to the complications from the boundary
values. Hence in this work we refrain from adopting the full Nambu
bracket for the shallow water equations.

We now turn to the discretization of the trilinear bracket
\eqref{eq:80}. The goal here is to preserve the skew-symmetry of the
trilinear bracket. This is much harder than the case of the bilinear
bracket, because there are more arguments present at the same
time. However, with the skew-symmetry property of the trilinear bracket
\eqref{eq:80} and its invariance under even permutation, one can
easily derive the relation
\begin{multline}
  \label{eq:83}
  \{F,\,H,\,Z\}_{\zeta\zeta\zeta} =
  \dfrac{1}{6}
  \left(\{F,\,H,\,Z\}_{\zeta\zeta\zeta} + 
    \{Z,\, F,\,H\}_{\zeta\zeta\zeta} + \{H,\,Z,\,F\}_{\zeta\zeta\zeta}\right.\\
    { } -\left.
  \{H,\,F,\,Z\}_{\zeta\zeta\zeta} - \{Z,\,H,\,F\}_{\zeta\zeta\zeta} -
  \{F,\,Z,\,H\}_{\zeta\zeta\zeta} \right). 
\end{multline}
Notes that the right-hand side of \eqref{eq:83} contains all the
possible permutations of $\{F,\,H,\,Z\}_{\zeta\zeta\zeta}$, with the even
ones taking the positive signs and the odd ones taking the negative
signs. This relation is trivial analytically, but numerically, the
expression on the right-hand has 
a crucial advantage compared with the one on the left-hand side,
namely, no matter how each individual term on the  right-hand side is
discretized, the whole expression will remain skew-symmetric, as
long as the permutations are preserved. This observation was
already made by Salmon in
\cite{Salmon2005-pd}. Hence, we will 
construct a discretization of the expression on the right-hand side of
\eqref{eq:83}, as an approximation to the trilinear bracket
\eqref{eq:80} and the Poisson bracket component \eqref{eq:27a}.

The discrete Hamiltonian will remain the same as before, as in
\eqref{eq:35}. The potential enstrophy \eqref{eq:79a} is approximated
by its discrete analogue
\begin{equation}
  \label{eq:84}
  Z = \int_\M \dfrac{1}{2} \phi_h q_h^2 d{\mathbf x}.
\end{equation}
As shown already in the previous section, this quantity has functional
derivatives
\begin{equation*}
  \dfrac{\delta Z}{\delta \phi_h} = -\dfrac{1}{2} q_h^2,\qquad
  \dfrac{\delta Z}{\delta \zeta_h} = q_h,\qquad
  \dfrac{\delta Z}{\delta \gamma_h} = 0.
\end{equation*}

The trilinear bracket $\{F,\,H,\,Z\}_{\zeta\zeta\zeta}$, as a
single component on the right-hand side of \eqref{eq:83}, is
approximated by
\begin{equation*}
  \{F,\,H,\,Z\}_{h,\zeta\zeta\zeta} \equiv 2\int_\M \hat
  Z_{\zeta_h}\nabla^\perp_h \tilde{F_{\zeta_h}}\cdot\nabla_h
  H_{\zeta_h} d{\mathbf x}. 
\end{equation*}
Thus, an discretization of  the trilinear bracket
$\{F,\,H,\,Z\}_{\zeta\zeta\zeta}$ itself, 
as well as the Poisson bracket component $\{F,\,H\}_{\zeta\zeta}$, can
be constructed by permuting the variables, using \eqref{eq:83},
\begin{multline}\label{eq:85}
  \{F,\,H\}_{h,\zeta\zeta} = \{F,\,H,\,Z\}_{h,\zeta\zeta\zeta} = {}\\
  \dfrac{1}{3}\left( \int_\M \hat
  Z_{\zeta_h}\nabla^\perp_h \tilde{F_{\zeta_h}}\cdot\nabla_h
  H_{\zeta_h} d{\mathbf x} + \int_\M \hat
  H_{\zeta_h}\nabla^\perp_h \tilde{Z_{\zeta_h}}\cdot\nabla_h
  F_{\zeta_h} d{\mathbf x} + \right. \\
  \int_\M \hat
  F_{\zeta_h}\nabla^\perp_h \tilde{H_{\zeta_h}}\cdot\nabla_h
  Z_{\zeta_h} d{\mathbf x} 
   - \int_\M \hat
  Z_{\zeta_h}\nabla^\perp_h \tilde{H_{\zeta_h}}\cdot\nabla_h
  F_{\zeta_h} d{\mathbf x} \\
  \left. {} - \int_\M \hat
  H_{\zeta_h}\nabla^\perp_h \tilde{F_{\zeta_h}}\cdot\nabla_h
  Z_{\zeta_h} d{\mathbf x} - \int_\M \hat
  F_{\zeta_h}\nabla^\perp_h \tilde{Z_{\zeta_h}}\cdot\nabla_h
  H_{\zeta_h} d{\mathbf x} \right).
\end{multline}
The key to ensure the skew-symmetry of the discrete trilinear bracket
$\{\cdot,\,\cdot,\, \cdot\}_{h,\zeta\zeta\zeta}$ is the consistency in
the operation $\nabla_h\tilde{(\phantom{F})}$ when applied to different
quantities. This operation is defined in \eqref{eq:sgt}.
% After substuting the actual expressions for the functional derivatives
% of $H$ and $Z$, one obtains
% \begin{multline}\label{eq:85}
%   \{F,\,H\}_{h,\zeta\zeta} = \{F,\,H,\,Z\}_{h,\zeta\zeta\zeta} = {}\\
%   \dfrac{1}{3}\left( -\int_\M \hat
%   q_h \nabla^\perp_h \tilde{F_{\zeta_h}}\cdot\nabla_h
%   \psi_h d{\mathbf x} - \int_\M \hat
%   \psi_h\nabla^\perp_h \tilde{q_h}\cdot\nabla_h
%   F_{\zeta_h} d{\mathbf x}  \right. \\
%   {}-\int_\M \hat
%   F_{\zeta_h}\nabla^\perp_h \tilde{\psi_h}\cdot\nabla_h
%   q_h d{\mathbf x} 
%    + \int_\M \hat
%   q_h\nabla^\perp_h \tilde{\psi_h}\cdot\nabla_h
%   F_{\zeta_h} d{\mathbf x} \\
%   \left. {} + \int_\M \hat
%   \psi_h\nabla^\perp_h \tilde{F_{\zeta_h}}\cdot\nabla_h
%   q_h d{\mathbf x} + \int_\M \hat
%   F_{\zeta_h}\nabla^\perp_h \tilde{q_h}\cdot\nabla_h
%   \psi_h d{\mathbf x} \right)
% \end{multline}

This trilinear bracket $\{\cdot,\,\cdot,\,\cdot\}_{h,\zeta\zeta\zeta}$
replaces the bilinear bracket $\{\cdot,\,\cdot\}_{h,\zeta\zeta}$ in
the discrete Poisson bracket, i.e.
\begin{equation}
  \label{eq:89}
  \{F,\, H\}_h = \{F,\,H,\,Z\}_{h,\zeta\zeta\zeta} + \{F,\,
  H\}_{h,\gamma\gamma} + \{F, \,H\}_{h,\phi\zeta\gamma}. 
\end{equation}
The second and third components are as defined in \eqref{eq:42} and
\eqref{eq:43}, respectively. The numerical scheme is again obtained
from the single scalar evolution equation
\begin{equation}
  \label{eq:90}
  \dfrac{\p F}{\p t} = \left\{F, \, H\right\}_h.
\end{equation}

We note that the change from the bilinear bracket
$\{\cdot,\,\cdot\}_{h,\zeta\zeta}$ to the trilinear bracket
$\{\cdot,\,\cdot,\,\cdot\}_{h,\zeta\zeta\zeta}$ 
only affects part of the equation for $\zeta_h$, and the
equations for $\gamma_h$ and $\phi_h$ remain unchanged. We now derive
the new equation for $\zeta_h$. Let us set
\begin{equation*}
  F = \zeta_i.
\end{equation*}
Then, as shown before, we have
\begin{equation*}
  \dfrac{\delta F}{\delta \zeta_h} = \delta_{ih}.
\end{equation*}
Substituting this expression, together with the functional derivatives
of $H$ and $Z$, into \eqref{eq:85}, one obtains
\begin{multline}\label{eq:86}
  \{\zeta_i,\,H\}_{h,\zeta\zeta} = \{\zeta_i,\,H,\,Z\}_{h,\zeta\zeta\zeta} = {}\\
  \dfrac{1}{3}\left( -\int_\M \hat
  q_h \nabla^\perp_h \tilde{\delta_{ih}}\cdot\nabla_h
  \psi_h d{\mathbf x} - \int_\M \hat
  \psi_h\nabla^\perp_h \tilde{q_h}\cdot\nabla_h
  \delta_{ih} d{\mathbf x}  \right. \\
  {}-\int_\M \hat
  \delta_{ih}\nabla^\perp_h \tilde{\psi_h}\cdot\nabla_h
  q_h d{\mathbf x} 
   + \int_\M \hat
  q_h\nabla^\perp_h \tilde{\psi_h}\cdot\nabla_h
  \delta_{ih} d{\mathbf x} \\
  \left. {} + \int_\M \hat
  \psi_h\nabla^\perp_h \tilde{\delta_{ih}}\cdot\nabla_h
  q_h d{\mathbf x} + \int_\M \hat
  \delta_{ih}\nabla^\perp_h \tilde{q_h}\cdot\nabla_h
  \psi_h d{\mathbf x} \right).
\end{multline}
Applying the discrete integration by parts formulas \eqref{eq:vc3} and
\eqref{eq:vc4}, and thanks to the adjoint identities \eqref{eq:vc1}
and \eqref{eq:vc2}, and to the homogeneous Dirichlet boundary
conditions on the discrete streamfunction $\psi_h$, one reaches at
\begin{multline}\label{eq:87}
  \{\zeta_i,\,H,\,Z\}_{h,\zeta\zeta\zeta} = \\
  \dfrac{1}{6}\left[
    \tilde{\nabla_h\times\left(\tilde q_h \nabla_h\psi_h -
        \tilde\psi_h\nabla_h q_h\right)} +
    \nabla_h\cdot\left(\hat\psi_h \nabla_h^\perp \tilde q_h - \hat q_h
      \nabla^\perp_h \tilde\psi_h\right)\right]_i +{}\\
  \dfrac{1}{3}\left[\hat{\nabla^\perp_h \tilde q_h \cdot \nabla_h\psi_h -
    \nabla^\perp_h\tilde\psi_h \cdot\nabla_h q_h}\right]_i. 
\end{multline}

The expression above replaces the term $\{F,\,H\}_{h,\zeta\zeta}$ in
the formulation of the equation of $\zeta_h$. The equations for
$\gamma_h$ and $\phi_h$ remain unchanged. We summarize the full set of
equations, for $\phi_h$, $\zeta_h$ and $\gamma_h$, as follows,
\begin{equation}
  \label{eq:88}
  \left\{
    \begin{aligned}
      \dfrac{d}{dt} \phi_i =
      & -\left[\Delta_h \chi_h\right]_i,\\
      \dfrac{d}{dt}\zeta_i   =
  &   \dfrac{1}{6}\left[
    \tilde{\nabla_h\times\left(\tilde q_h \nabla_h\psi_h -
        \tilde\psi_h\nabla_h q_h\right)} +
    \nabla_h\cdot\left(\hat\psi_h \nabla_h^\perp \tilde q_h - \hat q_h
      \nabla^\perp_h \tilde\psi_h\right)\right]_i +{}\\
  &\dfrac{1}{3}\left[\hat{\nabla^\perp_h \tilde q_h \cdot \nabla_h\psi_h -
    \nabla^\perp_h\tilde\psi_h \cdot\nabla_h q_h}\right]_i 
 -
\left[\nabla_h\cdot\left(\hat 
    q_h\nabla_h\chi_h\right)\right]_i,\\
\dfrac{d}{dt}\gamma_i =
  &\dfrac{1}{2}\left[\tilde{\nabla_h\times\left(\hat q_h 
    \nabla_h\chi_h\right)}\right]_i - \dfrac{1}{2}\left[
    \nabla_h\cdot \left(\hat q_h 
    \nabla^\perp_h\tilde\chi_h\right) \right]_i +
\left[\nabla_h\cdot\left(\hat q_h\nabla_h\psi_h\right)\right]_i\\
&- \left[\Delta_h
  \Phi_h\right]_i - \dfrac{1}{4|A_i|}\left(\hat q_{e_1} (\chi_i
  -\chi_{i_1}) + \hat q_{e_2} (\chi_{i_2} - \chi_i)\right).
    \end{aligned}\right.
\end{equation}
The term preceded by $1/4|A_i|$ in the equation for $\gamma_i$ only
appears for boundary cells ($i\in\BC$). 

% eqn-set for presentation slides
% \begin{equation*}
%   \LEFTRIGHT\{.{
%     \begin{aligned}
%       \dfrac{d}{dt} \phi_i =
%       & -\left[\Delta_h \chi_h\right]_i,\\
%       \dfrac{d}{dt}\zeta_i   =
%   &   \dfrac{1}{6}\left[
%     \tilde{\nabla_h\times\left(\tilde q_h \nabla_h\psi_h -
%         \tilde\psi_h\nabla_h q_h\right)} +
%     \nabla_h\cdot\left(\hat\psi_h \nabla_h^\perp \tilde q_h - \hat q_h
%       \nabla^\perp_h \tilde\psi_h\right)\right]_i +{}\\
%   &\dfrac{1}{3}\left[\hat{\nabla^\perp_h \tilde q_h \cdot \nabla_h\psi_h -
%     \nabla^\perp_h\tilde\psi_h \cdot\nabla_h q_h}\right]_i 
%  -
% \left[\nabla_h\cdot\left(\hat 
%     q_h\nabla_h\chi_h\right)\right]_i,\\
% \dfrac{d}{dt}\gamma_i =
%   &\dfrac{1}{2}\left[\tilde{\nabla_h\times\left(\hat q_h 
%     \nabla_h\chi_h\right)}\right]_i - \dfrac{1}{2}\left[
%     \nabla_h\cdot \left(\hat q_h 
%     \nabla^\perp_h\tilde\chi_h\right) \right]_i +
% \left[\nabla_h\cdot\left(\hat q_h\nabla_h\psi_h\right)\right]_i\\
% &- \left[\Delta_h
%   \Phi_h\right]_i - \dfrac{1}{4|A_i|}\left(\hat q_{e_1} (\chi_i
%   -\chi_{i_1}) + \hat q_{e_2} (\chi_{i_2} - \chi_i)\right).
%     \end{aligned}}
% \end{equation*}

A discrete analogue of Theorem \ref{thm:evolv-analytic} or Theorem
\ref{thm:evolv} can also be
established. 
\begin{theorem}\label{thm:evolv1} Let $\phi_h$, $\zeta_h$, and
  $\gamma_h$ be variables that evolve according to equation
  \eqref{eq:90}. Then
  an arbitrary functional $F=F(\phi_h,\,\zeta_h,\,\gamma_h)$  of
  the discrete state variables also evolves according to the equation
  \eqref{eq:90}. 
\end{theorem}

The proof of this theorem proceeds in a similar way as the proof of
Theorem \ref{thm:evolv}. We simply points that the new trilinear
bracket $\{\cdot,\,\cdot,\,\cdot\}_{h,\zeta\zeta\zeta}$ will not pose
any new difficulty, thanks to the homogeneous Dirichlet boundary
conditions assumed on the discrete streamfunction $\psi_h$.

As a consequence of Theorem \ref{thm:evolv1}, and thanks to the
skew-symmetry property of the discrete Poisson bracket \eqref{eq:89}
(and, indeed, of each of its components), the total energy $H$ is
conserved.

We now examine the conservation of quantities of the form
of \eqref{eq:75}. It is clear that the mass (when $G(q_h) = 1$) is
conserved under this new scheme.
When $G(q_h) = q_h$, $C = \int_\M \phi_h q_hd{\mathbf x} = \int_\M
(f+\zeta_h)d{\mathbf x}$ is the circulation. It has functional derivatives
\begin{equation*}
  \dfrac{\delta C}{\delta \phi_h} = 0,\qquad
  \dfrac{\delta C}{\delta \zeta_h} = 1,\qquad
  \dfrac{\delta C}{\delta \gamma_h} = 0.
\end{equation*}
Substituting these functional derivatives into the trilinear bracket
\eqref{eq:85}, one has
\begin{equation*}
 \{C,\,H,\,Z\}_{h,\zeta\zeta\zeta} = {}\\
  \dfrac{1}{3}\left(
  \int_\M 
  \nabla^\perp_h \tilde{H_{\zeta_h}}\cdot\nabla_h
  Z_{\zeta_h} d{\mathbf x}  - \int_\M 
  \nabla^\perp_h \tilde{Z_{\zeta_h}}\cdot\nabla_h
  H_{\zeta_h} d{\mathbf x} \right)
\end{equation*}
With the aide of the integration by parts formulas \eqref{eq:vc3} and
\eqref{eq:vc4}, and Lemmas \ref{lem:nondivergent} and \ref{lem:irrot},
one can show that both integrals in the above vanish. it has also been
shown in the previous section that both $\{F,\,H\}_{h,\gamma\gamma}$
and $\{F,\,H\}_{h,\phi\zeta\gamma}$ vanish, when $F$ is taken as the
total circulation. Hence the total circulation remains conserved under
the new scheme.

When $G(q_h) = q_h^2/2$, $C = \displaystyle\int_\M \phi_h q^2_h /2 d{\mathbf x}$ is equal to
$Z$, the 
potential enstrophy. It has been shown in the previous section that
both bilinear brackets $\{F,\,H\}_{h,\gamma\gamma}$ and
$\{F,\,H\}_{h,\phi\zeta\gamma}$ vanish when $F$ is taken as the
potential enstrophy. The new trilinear bracket
$\{F,\,H,\,Z\}_{h,\zeta\zeta\zeta}$ also vanishes when $F=Z$, thanks
to its skew-symmetry property. Thus, the potential enstrophy is
conserved under this new scheme.

\section{Linear dispersive analysis}\label{sec:linear-analysis}
Here we consider an infinite domain with a constant Coriolis force
$f_0$, and no bottom topography (i.e.~$b=0$). We also assume that
variations in the streamfunction $\psi_h$ 
and the velocity potential $\chi_h$ are small, and that the variations in
the fluid thickness $\phi_h$ is also small compared to its average
$\bar\phi$. Setting $\hat\phi_h=\bar\phi + \phi'_h$ in the definition
of the geopotential $\Phi_h$ in $\eqref{eq:36}_1$ and dropping the
quadratic terms, one obtains
\begin{equation}
  \label{eq:36-1}
  \Phi_h =  g(\bar\phi+ \phi'_h).
\end{equation}
Setting $\hat\phi_h=\bar\phi + \phi'_h$ in the last two equations of
\eqref{eq:36}, 
and then dropping the quadratic terms which are small, one obtains
\begin{equation}
  \label{eq:36b}\left\{
    \begin{aligned}
 &\bar\phi^{-1} \Delta_h \psi_h = \zeta_h,\\
 &\bar\phi^{-1} \Delta_h \chi_h = \gamma_h.
    \end{aligned}\right.
\end{equation}
Replacing $\phi_h$ by $\bar\phi+\phi'_h$ in \eqref{eq:27-3}, one can
also split the PV $q_h$ into a constant dominant part and a transient
part,
\begin{equation*}
  q_h = \dfrac{f_0}{\bar\phi} + q'_h.
\end{equation*}
Substituting this expression and \eqref{eq:36-1} into \eqref{eq:58},
and again dropping 
the quadratic nonlinear terms, one obtains
\begin{equation}
  \label{eq:58a}
  \left\{
    \begin{aligned}
      \dfrac{d}{dt} \phi'_h =
      & -\left[\Delta_h \chi_h\right]_i,\\
      \dfrac{d}{dt}\zeta_i   =
  & -\dfrac{f_0}{\bar\phi}\Delta_h \chi_h,\\
\dfrac{d}{dt}\gamma_h =
  &
\dfrac{f_0}{\bar\phi}\Delta_h \psi_h - g\Delta\phi'_h.
    \end{aligned}\right.
\end{equation}
Combining \eqref{eq:36b} with \eqref{eq:58a}, one finally obtains the
linearized version of the scheme \eqref{eq:58} as
\begin{equation}
  \label{eq:58b}
  \left\{
    \begin{aligned}
      \dfrac{d}{dt} \phi'_h =
      & -\bar\phi\gamma_h,\\
      \dfrac{d}{dt}\zeta_i   =
  & -f_0 \gamma_h,\\
\dfrac{d}{dt}\gamma_h =
  &f_0 \zeta_h - g\Delta\phi'_h.
    \end{aligned}\right.
\end{equation}
The linearized scheme is identical to the Z-grid scheme of
\cite{Randall1994-vu}, which indicates that the scheme \eqref{eq:58}
possess the same dispersive wave relations as the Z-grid scheme.

Analysis on the enstrophy-conserving scheme \eqref{eq:88} leads to the
same conclusion.

\section{Conclusion}\label{sec:conclusion}
This work develops conservative finite volume schemes, with optimal
dispersive wave relations, for large-scale geophysical flows on
unstructured meshes over bounded or unbounded domains. It uses the
vorticity-divergence formulation of the dynamical equations for the
sake of the dispersive wave relations. The Hamiltonian approach is
followed, which seeks to preserve the skew-symmetric structures within
the system. Two numerical schemes are developed with the first one
conserving the total energy only, and the other slightly more complex
one conserving both the energy and (potential) enstrophy. Both schemes
are also shown to possess the same optimal dispersive wave relations as
those of the Z-grid scheme of \cite{Randall1994-vu}.

For numerical schemes based on the vorticity-divergence formulations,
the issue of boundary conditions is very tricky. For fluid flows, the
classical boundary conditions, e.g.~the no-flux or no-slip conditions,
are typically imposed on the velocity variables, and they have no
counterparts for the vorticity/divergence variables. Artificially
imposed boundary conditions on the vorticity/divergence can lead to
either an under-specified, and over-specified, or an inconsistent
system, with serious stability and accuracy implications. Numerous
efforts have been spent on this issue (\cite{Thom1933-fg,
  Orszag1986-bm, E1996-te}). In this work, the issue of boundary
conditions are addressed at three different levels. At the first
level, the boundary conditions that are mandated by the analytical
system are imposed on the discrete streamfunction and velocity
potential directly; see equation \eqref{eq:36aa}. At the next level,
it is made sure that various 
differential operators are discretized consistently when it comes near
the boundary; see \eqref{eq:gradt} and \eqref{eq:sgt}. Lastly, at the
highest level on the vorticity and 
divergence variables, their values on the boundary are not imposed,
but are simply evolved along with the other prognostic variables, as
dictated by the discrete canonical equations. At this last level, our
approach bears some resemblance to, but also substantial difference from
that of \cite{Bauer2018-uv}, which adds a prognostic equation for the
potential vorticity on the boundary, in addition to the momentum
equations in the interior of the domain.

The current work bears the greatest affinity to that of Eldred and
Randall (\cite{Eldred2017-ji}), and that of Bauer and Cotter
(\cite{Bauer2018-uv}). But there are significant differences as
well. Both our work and \cite{Eldred2017-ji} start from the
vorticity-divergence formulation and the Hamiltonian approach. But
\cite{Eldred2017-ji} uses the Nambu brackets throughout the whole
system to achieve the conservation of enstrophy, and it only considers
the global sphere. The current work uses the Nambu bracket only in the
$\zeta\zeta$-component of the Poisson bracket, and it takes care of
the boundary conditions to ensure that the Poisson/Nambu brackets
remain skew-symmetric. The current work share with \cite{Bauer2018-uv}
the common goal of deriving energy and enstrophy conserving schemes
over bounded domains. The current work differs in the discretization
methodology (FV {\it vs} FEM), the staggering techniques (Z-grid {\it vs}
C-grid), and the 
resulting dispersive wave relations (Z-grid {\it vs} C-grid). 
  
While the schemes developed in this work are shown to possess
excellent conservative properties and dispersive wave relations, other
questions regarding accuracy, dynamical behaviors, etc., are left
untouched. Also left open are the crucial questions on the real
advantages of a doubly-conservative scheme over, say, an
energy-conserving only scheme, and how a doubly-conservative scheme
should be utilized in practice. These questions will be explored in
the second, numerical installment of this project.

\newpage
\appendix
\section{Specifications of the mesh}\label{s:mesh}
\begin{figure}[h]
  \centering
  \includegraphics[width=4.5in]{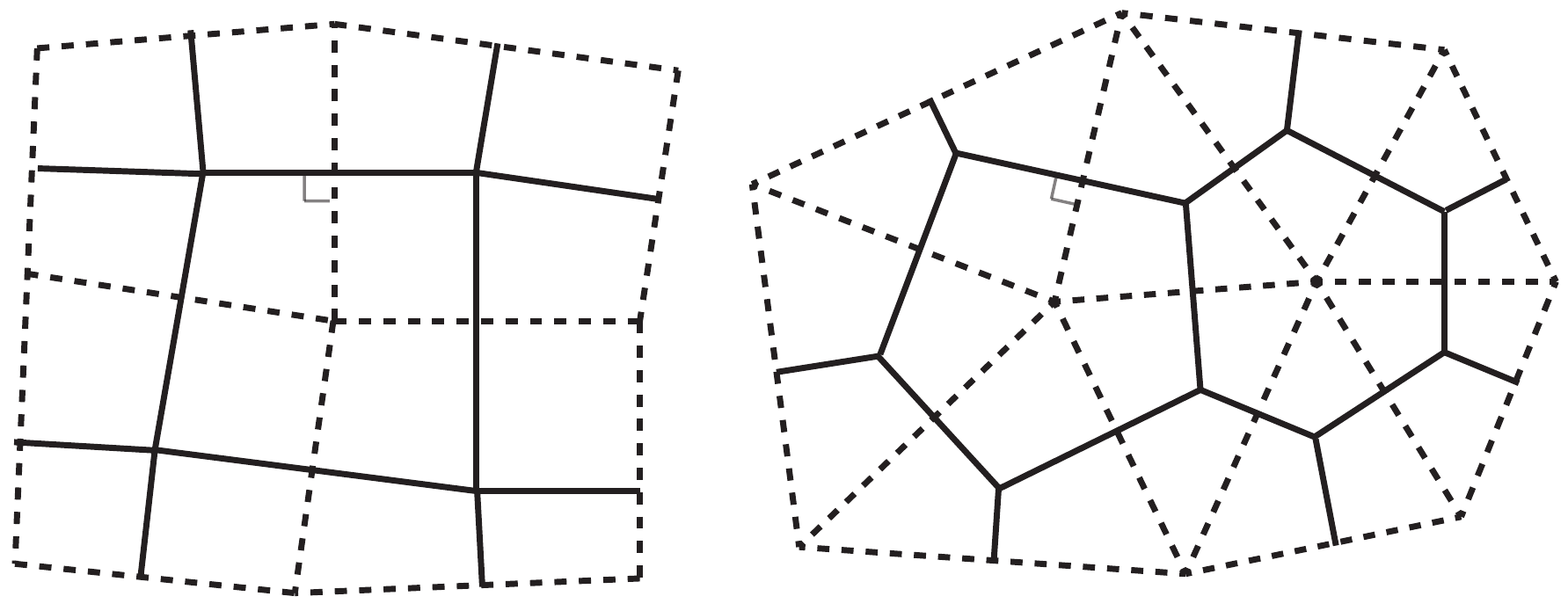}
  \caption{Generic dual meshes, with the domain boundaries passing
    through the primary cell centers. Left: a generic quadrilateral
    dual mesh; right: a generic Delaunay-Voronoi dual mesh.}
  \label{fig:quad-dv}
\end{figure}

\begin{table}[h]
\centering
\ra{1.3}
\caption{Sets of basic mesh elements.}
\vspace{4mm}
\begin{tabular}{@{}ll@{}}\toprule
Set & Definition\\
\midrule
$\IC$ & Set of interior cells\\
$\BC$ & Set of boundary cells\\
$\IE$ & Set of interior edges\\
$\BE$ & Set of boundary edges\\
$\V$ & Set of vertices\\
\bottomrule
\end{tabular}
\label{ta2}
\end{table}

Our approximation of the function space is based on discrete meshes
that consist of polygons. To avoid potential technical issues with the boundary, we
shall assume that the domain $\Omega$ itself is polygonal. We make use
of a pair of staggered meshes, with one called primary and the other called
dual. The meshes consist of polygons, called cells, of arbitrary
shape, but conforming to the requirements to be specified. The centers
of the cells on the primary mesh are the vertices of the cells on the
dual mesh, and vice versa. The edges of the primary cells intersect
{\it orthogonally} with the edges of the dual cells. The line segments
of the boundary $\partial\Omega$ pass through the centers of the
primary cells that border the boundary. Thus the primary cells on the
boundary are only partially contained in the domain. 
{Two examples of this mesh type are shown in Figure
  \ref{fig:quad-dv}.}

% Shown in Figure \ref{fig:quad-dv} are two common types of staggered
% grids: a quadrilateral-quadrilateral staggered grid (left), and a
% Delaunay-Voronoi tessellation (right). 
% Simple diagrams
% of such dual meshes are shown in Figures \ref{fig:quad-dv}. 
% {\bf Introduce quads and Delauny-Voronoi}. 

\begin{figure}[h]
  \centering
  {\includegraphics[width=3in]{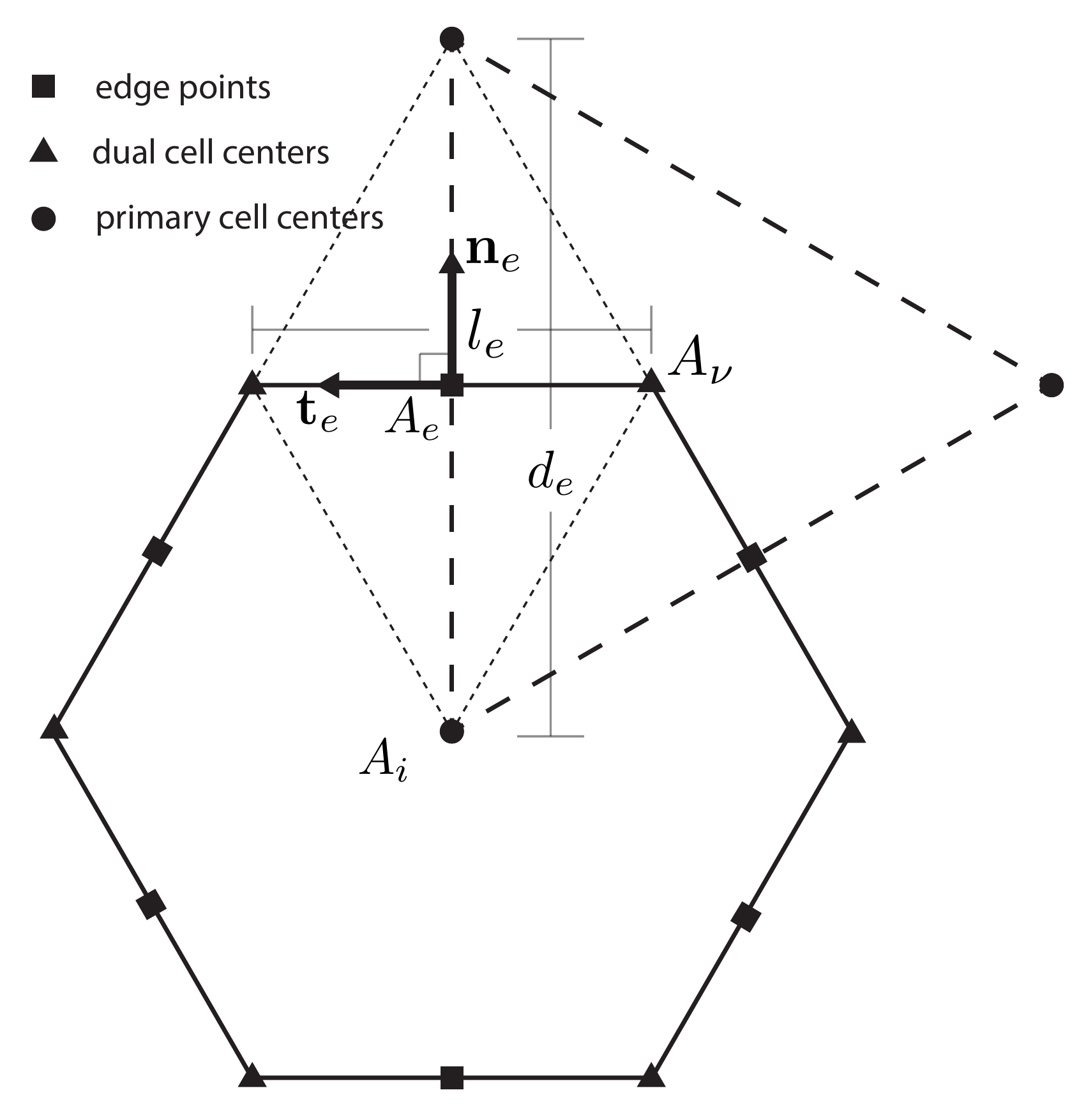}}
  \caption{Illustration of the notations.}
  \label{fig:notations}
\end{figure}

\begin{table}[h]
\centering
\ra{1.3}
\caption{Sets of elements defining the connectivity of a unstructured
  dual grid.}
\vspace{4mm}
\begin{tabular}{@{}ll@{}}\toprule
Set & Definition\\
\midrule
$\EC(i)$ & Set of edges defining the boundary of primary cell $A_i$\\
$\VC(i)$ & Set of dual cells that form the vertices primary cell $A_i$\\
$\CE(e)$ & Set of primary cells boarding edge $e$\\
$\VE(e)$ & Set of dual cells boarding edge $e$\\
$\CV(\nu)$ & Set of primary cells that form vertices of dual cell $D_\nu$\\
$\EV(\nu)$ & Set of edges that define the boundary of dual cell $D_\nu$\\
\bottomrule
\end{tabular}
\label{ta1}
\end{table}

In order to construct function spaces on this type of meshes, some
notations are in order, for which we follow the conventions made in
\cite{Ringler2010-sm,Chen2013-fa}. As shown in the diagram in Figure
\ref{fig:notations}, the primary cells are denoted as $A_i,\, 1\leq
i\leq N_c + N_{cb}$, where $N_c$ denotes the number of cells that are
in the interior of the domain, and $N_{cb}$ the number of cells that
are on the boundary. 
% We assume the cells are numbered so that $A_i$
% with $1\leq i\leq N_c$ refer to interior cells.  
The dual cells, which
all lie inside the domain, are denoted as $A_\nu,\,1\leq \nu\leq
N_v$. 
The area of $A_i$ (resp.~$A_\nu$) is denoted as $|A_i|$
(resp.~$|A_\nu|$). 
% When no 
% confusion should arise, we also use $A_i$ and $A_\nu$ to denote the
% areas of the primary cells and dual cells, respectively.
Each primary
cell edge corresponds to a distinct dual cell edge, and vice
versa. Thus the primary and dual cell edges share a common index $e,\,
1\leq e\leq N_e+N_{eb}$, where $N_e$ denotes the number of edge pairs
that lie entirely in the interior of the domain, and $N_{eb}$ the
number of edge pairs on the boundary, i.e., with dual cell edge
aligned with the boundary of the domain. 
% Again, we assume that $1\le
% e\le N_e$ refer to interior edges.
Upon an edge pair $e$, the distance between the two
primary cell centers, which is also the length of the corresponding
dual cell edge, is denoted as $d_e$, while the distance between the
two dual cell centers, which is also the length of the corresponding
primary cell edge, is denoted as $l_e$. 
These two edges form the diagonals of a diamond-shaped region, whose
vertices consist of the two neighboring primary cell centers and the
two neighboring dual centers. The diamond-shaped region is also
indexed by $e$, and will be referred to as $A_e$.
The Euler formula for planar
graphs states that the number of primary cell centers $N_c + N_{cb}$, the
number of vertices (dual cell centers) $N_v$, and the number of
primary or dual cell edges $N_e + N_{eb}$ must satisfy the relation
\begin{equation}
\label{eq:g44}
  N_c + N_{cb} + N_v = N_e + N_{eb} + 1.
\end{equation}
The connectivity information of the unstructured staggered meshes is
provided by six {\it sets of elements} defined in Table \ref{ta1}. 

For each edge pair, a unit vector $\nb_e$, normal to the primary cell
edge, is specified. A second unit vector $\tb_e$ is defined as
\begin{equation}
\label{eq:g45}
  \tb_e = \kb\times\nb_e,
\end{equation}
with $\kb$ standing for the upward unit vector. Thus $\tb_e$ is
orthogonal to the dual cell edge, but tangent to the primary cell
edge, and points to the vertex on the left side of $\nb_e$. For each
edge $e$ and for each $i\in \CE(e)$ (the set of cells on edge $e$, see
Table \ref{ta1}), we define the direction indicator
\begin{equation}
\label{eq:g46}
  n_{e,i} = \left\{
  \begin{aligned}
    1& & &\phantom{sss}\textrm{if }\nb_e\textrm{ points away from primary cell
    }A_i,\\
    -1& &  &\phantom{sss}\textrm{if }\nb_e\textrm{ points towards primary cell
    }A_i,\\
  \end{aligned}\right.
\end{equation}
and for each $\nu\in \VE(e)$,
\begin{equation}
\label{eq:g47}
  t_{e,\nu} = \left\{
  \begin{aligned}
    1& & &\phantom{sss}\textrm{if }\tb_e\textrm{ points away from dual cell
    }A_\nu,\\
    -1& &  &\phantom{sss}\textrm{if }\tb_e\textrm{ points towards dual cell
    }A_\nu.\\
  \end{aligned}\right.
\end{equation}

For this study, we make the following regularity assumptions on the
meshes. We assume that the diamond-shaped region $A_e$ is actually
convex. In other words, the intersection point of each edge pair falls
inside each of the two edges. We also assume that the meshes are
quasi-uniform, in the sense that there exists $h>0$ such that, for 
each edge $e$,
\begin{equation}
\label{eq:g48}
  mh\leq l_e,\,d_e \leq Mh
\end{equation}
for some fixed constants $(m,\,M)$ that are independent of the
meshes. 
The staggered dual meshes are thus designated by $\mathcal{T}_h$.
{For the convergence analysis, it is assumed in
  \cite{Chen2016-gl} that, for each edge pair $e$, the primary cell
  edge nearly bisect 
  the dual cell edge, and miss by at most $O(h^2)$. This assumption is
  also made here for the error analysis. Generating meshes
  conforming to this requirement on irregular domains, i.e.~domains
  with non-smooth boundaries or domains on surfaces, can be a
  challenge, and will be addressed elsewhere. But we point out that, on
  regular domains with smooth boundaries, this type of meshes can be
  generated with little extra effort in addition to the use of
  standard mesh generators, such as the 
  centroidal Voronoi tessellation algorithm (\cite{Du1999-th,Du2002-lf,Du2003-gn})}.  

\section{Specifications of the discrete averaging and differential
  operators}\label{s:ddo}

In this work, a non-accent symbol, such as $\psi$, usually designates
a discrete scalar field defined at the (primal) cell centers, while
$\widetilde{\phantom{\psi}}$ on the top, as in $\widetilde\psi_h$,
designates a discrete scalar field at 
the vertices, i.e.~dual cell centers, and $\widehat{\phantom{\psi}}$
on the top,
as in $\hat\psi_h$, designates a discrete scalar  field at the edges.

We use $\tilde{[\cdot]}$ to designate the mappings between primal cell
centers and dual cell centers, in both directions. We let $\psi_h$ be
a discrete scalar field defined at cell centers. Then these mappings
can be defined as follows,
\begin{subequations}\label{eq:cv}
\begin{align}
  &\tilde{[\psi_h]} = \sum_{\nu\in\V} \tilde\psi_\nu\chi_\nu, & 
   &\textrm{ with }
  \tilde\psi_\nu = \dfrac{1}{A_\nu} \sum_{i\in CV(\nu)} \psi_i
    A_{i,\nu},\label{eq:cv1}\\ 
  &\tilde{[\tilde\psi_h]} = \sum_{i\in\C} \psi_i\chi_i, &
    &\textrm{ with }
  \psi_i = \dfrac{1}{A_i} \sum_{\nu\in VC(i)} \tilde\psi_\nu
    A_{i,\nu}.\label{eq:cv2}
\end{align}
\end{subequations}
Generally, on unstructured meshes, the composition of these two
mappings is not the identity mapping.

We use $\hat{[\cdot]}$ to designate the mappings between primal cell
centers and cell edges, in both directions. Again, with $\psi_h$ being
a discrete scalar field at cell centers, these mappings
are specified  as follows,
\begin{subequations}\label{eq:ce}
\begin{align}
  &\hat{[\psi_h]} = \sum_{\nu\in\V} \hat\psi_e\chi_e, & &\textrm{
    with }
  \hat\psi_e = \dfrac{1}{1} \sum_{i\in CE(e)} \psi_i,\label{eq:ce1}\\  
  &\hat{[\hat\psi_h]} = \sum_{i\in\C} \psi_i\chi_i, &
    &\textrm{ with }
  \psi_i = \dfrac{1}{2A_i} \sum_{\nu\in VC(i)} \hat\psi_e
    A_{e}.\label{eq:ce2}
\end{align}
\end{subequations}
Generally, on unstructured meshes, the composition of these two
mappings is not the identity mapping.

The
discrete gradient operator $\nabla_h$ on a cell-centered scalar field
$\psi_h$ is defined as 
\begin{equation}
  \nabla_h \psi_h = \sum_{e\in\E} [\nabla_h\psi_h]_e \chi_e
  \nb_e,\qquad \textrm{with}\qquad  
  [\nabla_h\psi_h]_e = \dfrac{-1}{d_e}\sum_{i\in \CE(e)}\psi_i
  n_{e,i}.\label{eq:gradc} 
\end{equation}
Clearly, $[\nabla_h\psi_h]_e$ is an approximation of
$\nabla\psi\cdot\nb$ at cell edge $e$. Similarly, the discrete
gradient operator $\nabla_h$ on a triangle-centered scalar field
$\tilde\psi_h$ is defined as 
\begin{equation}\label{eq:gradt} 
  \nabla_h \tilde\psi_h = \sum_{e\in\E} [\nabla_h\tilde\psi_h]_e \chi_e
  \tb_e,\qquad \textrm{with}\qquad  
  [\nabla_h\tilde\psi_h]_e = \left\{
  \begin{aligned}
  &\dfrac{-1}{l_e}\sum_{\nu\in \VE(e)}\tilde\psi_\nu
  t_{e,\nu}, & &\textrm{ if } e\in\IE,\\
  &\dfrac{t_{e,\nu}(\hat\psi_e - \tilde\psi_\nu)}{l_e}, & &\textrm{ if
  } e\in\BE.
  \end{aligned}\right.
\end{equation}
Clearly, $[\nabla_h\tilde\psi_h]_e$ is an approximation of
$\nabla\psi\cdot\tb$ at cell edge $e$. 

The discrete skew gradient operator $\nabla^\perp_h$ on a
cell-centered scalar field $\psi_h$:
\begin{equation} 
  \nabla_h^{\perp} \psi_h = \sum_{e\in\E}
  [\nabla_h^{\perp}\psi_h]_e \chi_e \tb_e,\qquad\textrm{with}\qquad 
  [\nabla_h^{\perp}\psi_h]_e = \dfrac{-1}{d_e}\sum_{i\in
    \CE(e)}\psi_i n_{e,i}.\label{eq:sg} 
\end{equation}
$[\nabla_h^{\perp}\psi_h]_e$ is a discretization of
$\nabla^\perp\psi\cdot\tb \equiv \p\psi/\p n$ at cell edge e.

The discrete skew gradient operator $\nabla^\perp_h$ on a
triangle-centered scalar field $\tilde\psi_h$:
\begin{equation}\label{eq:sgt} 
  \nabla_h^{\perp} \tilde\psi_h = \sum_{e\in\E}
  [\nabla_h^{\perp}\tilde\psi_h]_e \chi_e \nb_e,
  \end{equation}
  with 
\begin{equation*}
  [\nabla_h^{\perp}\tilde\psi_h]_e =\left\{
  \begin{aligned}
  &\dfrac{1}{l_e}\sum_{\nu\in
    \VE(e)}\tilde\psi_\nu t_{e,\nu}, & &\textrm{ if } e\in\IE,\\
  &\dfrac{t_{e,\nu} (\tilde\psi_\nu - \hat\psi_e)}{l_e}\textrm{ for }
  \nu\in\VE(e), & &\textrm{ if } e\in\BE.
  \end{aligned}\right.
\end{equation*}
where $[\nabla_h^{\perp}\tilde\psi_h]_e$ is a discretization of
$\nabla^\perp\psi\cdot\nb\equiv -\p\psi/\p\tau$ at cell edge e.

% The situation on the boundary requires some comments. With
% each boundary edge, only one vertex is associated. Hence on a
% boundary edge $e$, the specifications of the discrete gradient
% operator $\nabla_h$ and the discrete skew gradient operator
% $\nabla^\perp_h$  have the forms 
% \begin{align*}
%   [\nabla_h\tilde\psi_h]_{e\textrm{ on boundary}} &=
%   \dfrac{-1}{l_e}\tilde\psi_\nu t_{e,\nu},\\
%   [\nabla_h^{\perp}\tilde\psi_h]_{e\textrm{ on boundary}} &=
%   \dfrac{1}{l_e}\tilde\psi_\nu t_{e,\nu}.
%  \end{align*} 
% where $\nu$ is the single element in $\VE(e)$. This amounts to
% implicitly requiring that $\tilde\psi_h$ vanishes on the boundary.

We denote by $u_h$ a discrete vector field that is along the normal
direction at each cell edge, and by $v_h$ a discrete vector field that
is along the tangential direction at each cell edge; these discrete
vector fields can be expressed as
\begin{align*}
  u_h &= \sum_{e\in\E} u_e \chi_e \nb_e,\\
  v_h &= \sum_{e\in\E} v_e \chi_e \tb_e.
\end{align*}
On these discrete vector fields various discrete differential
operators can be defined, using the discrete versions of the
divergence theorem or the Green's theorem.

The discrete divergence operator $\nabla_h\cdot(\cdot)$ on a discrete
normal vector field $u_h$:
\begin{equation}
  \label{eq:div}
  \nabla_h\cdot u_h = \sum_{i\in\C}\left[\nabla_h \cdot
    u_h\right]_i\chi_i,\qquad\textrm{with}\qquad
  \left[\nabla_h \cdot u_h\right]_i = \dfrac{1}{A_i}\sum_{e\in \EC(i)}u_e l_e n_{e,i}.
\end{equation}
It is clear that $\nabla_h\cdot u_h$ is a discrete scalar field on the
primal mesh.
It is worth noting that, {\itshape on partial cells on the boundary,
  the 
summation on the right-hand side only includes fluxes across the edges
that are inside the domain and the partial edges that intersect with
the boundary, and this amounts to imposing a no-flux
condition across the boundary.}

The discrete divergence operator $\nabla_h\cdot(\cdot)$ on a discrete
tangential vector field $v_h$:
\begin{equation}
  \label{eq:divt}
  \nabla_h\cdot v_h = \sum_{\nu\in\V}\left[\nabla_h \cdot
    v_h\right]_\nu\chi_\nu,\qquad\textrm{with}\qquad
  \left[\nabla_h \cdot v_h\right]_\nu = \dfrac{1}{A_\nu}\sum_{e\in
    \EV(\nu)}v_e d_e t_{e,\nu}. 
\end{equation}
It is clear that $\nabla_h\cdot u_h$ is a discrete scalar field on the
dual mesh.

The discrete curl operator $\nabla_h\times(\cdot)$ on a discrete
normal vector field:
\begin{equation}
\label{eq:curl}
  \nabla_h\times u_h = \sum_{\nu\in\V}\left[\nabla_h \times
    u_h\right]_\nu\chi_\nu,\quad\textrm{with}\quad
  \left[\nabla_h \times u_h\right]_\nu =
  \dfrac{-1}{A_\nu}\sum_{e\in \EV(\nu)}u_e d_e t_{e,\nu}. 
\end{equation}
Thus, the image of the discrete curl operator
$\nabla_h\times(\,)$ on each $u_h$ is a discrete scalar field on
the dual mesh.

The discrete curl operator $\nabla_h\times(\cdot)$ on a discrete
tangential vector field:
\begin{equation}
\label{eq:curlt}
  \nabla_h\times v_h = \sum_{i\in\C}\left[\nabla_h \times
    v_h\right]_i \chi_i,\quad\textrm{with}\quad
  \left[\nabla_h \times v_h\right]_i =
  \dfrac{1}{A_i}\sum_{e\in \EC(i)}v_e l_e n_{e,i}. 
\end{equation}
Thus, the image of the discrete curl operator
$\nabla_h\times(\,)$ on each $v_h$ is a discrete scalar field on
the primal mesh.
It is worth noting that, {\itshape on partial cells on the boundary,
  the 
summation on the right-hand side only includes currents along the edges
that are inside the domain and the partial edges that intersect with
the boundary, and this amounts to imposing a no-circulation/no-slip 
condition along the boundary.}

For a scalar field $\phi_h$ defined at cell centers, the discrete
Laplacian operator $\Delta_h$ can also be defined,
\begin{equation}
  \label{eq:d30}
  \Delta_h \phi_h := \nabla_h\cdot\left(\nabla_h\phi_h\right) \equiv
  \nabla_h \times\left(\nabla_h^\perp \phi_h\right).
\end{equation}

\section{Discrete vector calculus}\label{sec:dvc}
\begin{lemma}\label{lem:CEV-adj}
  \begin{align}
    &\left(\hat \psi_h,\,\hat \varphi_h\right) =
    \left(\psi_h,\,\hat{\hat\varphi}_h\right),\label{eq:vc1}\\ 
    &\left(\tilde \psi_h,\,\tilde \varphi_h\right) =
    \left(\psi_h,\,\tilde{\tilde\varphi}_h\right).\label{eq:vc2} 
  \end{align}
\end{lemma}
\begin{proof}
  We prove the equality \eqref{eq:vc1} first, using the specifications
  \eqref{eq:ce} for the remapping between center centers and cell
  edges, 
  \begin{align*}
    \left(\hat\psi_h, \hat\varphi_h\right)
    =& \sum_e \left(\dfrac{1}{2}\sum_{i\in\CE(e)}\psi_i \right)
       \hat\varphi_e A_e \\
    =& \dfrac{1}{2}\sum_i \psi_i \sum_{e\in\EC(i)} \hat\varphi_e A_e\\
    =& \sum_i\psi_i \left(\dfrac{1}{A_i}\sum_{e\in \EC(i)}
       \hat\varphi_e \dfrac{A_e}{2}\right) A_i\\
    =& \left(\psi_h,\, \hat{\hat\varphi}_h\right).
  \end{align*}

Similarly, for the equality \eqref{eq:vc2}, one uses the
specifications \eqref{eq:cv}, and proceeds,
  \begin{align*}
    \left(\tilde\psi_h,\, \tilde\varphi_h\right)
    =& \sum_\nu \left(\dfrac{1}{A_\nu}\sum_{i\in\CV(\nu)}\psi_i A_{i\nu}\right)
       \tilde\varphi_\nu A_\nu \\
  =& \sum_\nu \left(\sum_{i\in\CV(\nu)}\psi_i A_{i\nu}\right)
       \tilde\varphi_\nu \\
    =& \sum_i \psi_i \sum_{\nu\in\VC(i)} \tilde\varphi_\nu A_{i\nu}\\
    =& \sum_i\psi_i \left(\dfrac{1}{A_i}\sum_{\nu\in \VC(i)}
       \tilde\varphi_\nu A_{i\nu}\right) A_i\\
    =& \left(\psi_h,\, \tilde{\tilde\varphi}_h\right).
  \end{align*}

\end{proof}

\begin{lemma}\label{lem:integ-by-parts}
  Let $u_h$ be a discrete normal vector field, $v_h$ a discrete
  tangential vector field, $\psi_h$ a
  cell-centered discrete scalar field, and $\tilde\psi_h$ a
  triangle-centered discrete scalar field. Then
  the following relations hold,
  \begin{align}
    \label{eq:vc3}
    \left( u_h,\,\nabla_h\psi_h\right)_{0,h}
    &= -\dfrac{1}{2}\left(\nabla_h\cdot u_h,\,\psi_h\right)_{0,h},\\
    \left( u_h,\,\nabla_h^\perp\tilde\psi_h\right)_{0,h}
    &= -\dfrac{1}{2}\left(\nabla_h\times u_h,\,\tilde\psi_h\right)_{0,h}
       - \dfrac{1}{2}\sum_{e\in\BE}\hat\psi_e u_e d_e \sum_{\nu\in
      VE(e)} t_{e,\nu},\label{eq:vc4}\\
    \left( v_h,\,\nabla_h^\perp\psi_h\right)_{0,h}
    &= -\dfrac{1}{2}\left(\nabla_h\times
      v_h,\,\psi_h\right)_{0,h},\label{eq:vc3a}\\ 
    \left( v_h,\,\nabla_h\tilde\psi_h\right)_{0,h}
    &= -\dfrac{1}{2}\left(\nabla_h\cdot v_h,\,\tilde\psi_h\right)_{0,h}
       + \dfrac{1}{2}\sum_{e\in\BE}\hat\psi_e v_e d_e \sum_{\nu\in
      VE(e)} t_{e,\nu}.\label{eq:vc4a}
  \end{align}
\end{lemma}

\begin{lemma}\label{lem:nondivergent}
Assume that the domain $\M$ is simply connected, and $u_h$ is a
discrete normal vector field.  Then,
\begin{equation}
  \label{eq:vc5}
\nabla_h\cdot u_h = 0
\end{equation}
if and only if 
  \begin{equation}
    \label{eq:vc6}
    u_h = \nabla_h^\perp\tilde\psi_h,
  \end{equation}
for some   triangle-centered scalar field  $\tilde\psi_h$. 
\end{lemma}

\begin{lemma}\label{lem:irrot}
Assume that the domain $\M$ is simply connected, and $u_h$ is a
discrete normal vector field.  Then,
  \begin{equation}
    \label{eq:vc7}
\nabla_h\times u_h = 0
  \end{equation}
 if and only if there
  exists a cell-centered scalar field $\varphi_h$ such that
  \begin{equation}
\label{eq:vc8}
    u_h = \nabla_h\varphi_h.
  \end{equation}
\end{lemma}

\bibliographystyle{siamplain}
\bibliography{references}   % name your BibTeX data base

\end{document}